%% file: main_laplacekoenigs.tex
\documentclass[a4paper,twoside,11pt]{article}
\usepackage[utf8]{inputenc}
\usepackage[english]{babel}
\usepackage{amssymb, amsfonts, amsthm, amsmath}
\usepackage[hidelinks]{hyperref}
\usepackage{pgf,tikz} % for drawing vector graphics inside latex
\usetikzlibrary{math,arrows,calc,through}
\usepackage[margin=2.5cm]{geometry} % more realistic margin
\usepackage[mathscr]{euscript}% Calligraphy
\usepackage{psfrag}
\usepackage[percent]{overpic}% to add notes on Figures
\usepackage{graphicx,accents}
\usepackage{placeins}
\usepackage{caption} % so that hyperref in pdf navigation jumps to top of figures
\usepackage{mathtools} % so only referenced equations are numbered (with line below)
\usepackage{subcaption}
\usepackage{enumitem}
\usepackage{color}
\usepackage{bm} % bold math symbols
\usepackage[]{changepage} %for the margins in the statement environment

\numberwithin{equation}{section}

\theoremstyle{plain}
\newtheorem{theorem}{Theorem}[section]
\newtheorem{corollary}[theorem]{Corollary}

\newtheorem{lemma}[theorem]{Lemma}
\newtheorem{proposition}[theorem]{Proposition}

\theoremstyle{definition}
\newtheorem{definition}[theorem]{Definition}

\newtheorem{remark}[theorem]{Remark}

\theoremstyle{remark}

\newenvironment{statement}
{	\itshape
	\vspace{1mm}
	\begin{adjustwidth}{5mm}{5mm}
		\begin{center}
}
{ 
		\end{center}
	\end{adjustwidth}
	\vspace{1mm}
}

\newcommand{\N}{\mathbb{N}}

\newcommand{\Z}{\mathbb{Z}}
\newcommand{\R}{\mathbb{R}}

\newcommand{\RP}{{\mathbb{R}\mathrm{P}}}

\newcommand{\join}{\mathrm{join}}
\DeclareMathOperator{\cro}{cr} %cross-ratio
\newcommand{\pah}[1]{\raisebox{0pt}[0pt][0pt]{$#1^\vee_{\scalebox{0.6}[1]{$\bm{-}$}}$}}
\newcommand{\pav}[1]{#1^\vee_{\,\bm{\shortmid}}}
\newcommand{\pa}[1]{#1^\vee}

\newcommand{\p}{\varphi}

\renewcommand{\a}{\alpha}
\renewcommand{\b}{\beta}

\DeclareMathOperator{\North}{
	\begin{tikzpicture}[scale=1.4, line cap=round]%
		\draw (0,0) -- (0,1ex) -- (1ex,1ex) -- (1ex,0) -- (0,0);
		\fill (0ex,0.5ex) rectangle ++(1ex,0.5ex);
	\end{tikzpicture}\kern-1pt
}
\DeclareMathOperator{\East}{
	\begin{tikzpicture}[scale=1.4, line cap=round]
		\draw (0,0) -- (0,1ex) -- (1ex,1ex) -- (1ex,0) -- (0,0);
		\fill (0.5ex,0ex) rectangle ++(0.5ex,1ex);
	\end{tikzpicture}\kern-1pt
}
\DeclareMathOperator{\South}{
	\begin{tikzpicture}[scale=1.4, line cap=round]
		\draw (0,0) -- (0,1ex) -- (1ex,1ex) -- (1ex,0) -- (0,0);
		\fill (0ex,0ex) rectangle ++(1ex,0.5ex);
	\end{tikzpicture}\kern-1pt
}
\DeclareMathOperator{\West}{
	\begin{tikzpicture}[scale=1.4, line cap=round]
		\draw (0,0) -- (0,1ex) -- (1ex,1ex) -- (1ex,0) -- (0,0);
		\fill (0ex,0ex) rectangle ++(0.5ex,1ex);
	\end{tikzpicture}\kern-1pt
}
\DeclareMathOperator{\NE}{
	\begin{tikzpicture}[scale=1.4, line cap=round]
		\draw (0,0) -- (0,1ex) -- (1ex,1ex) -- (1ex,0) -- (0,0);
		\fill (0.5ex,0.5ex) rectangle ++(0.5ex,0.5ex);
	\end{tikzpicture}\kern-1pt
}
\DeclareMathOperator{\SE}{
	\begin{tikzpicture}[scale=1.4, line cap=round]
		\draw (0,0) -- (0,1ex) -- (1ex,1ex) -- (1ex,0) -- (0,0);
		\fill (0.5ex,0) rectangle ++(0.5ex,0.5ex);
	\end{tikzpicture}\kern-1pt
}
\DeclareMathOperator{\SW}{
	\begin{tikzpicture}[scale=1.4, line cap=round]
		\draw (0,0) -- (0,1ex) -- (1ex,1ex) -- (1ex,0) -- (0,0);
		\fill (0ex,0ex) rectangle ++(0.5ex,0.5ex);
	\end{tikzpicture}\kern-1pt
}
\DeclareMathOperator{\NW}{
\begin{tikzpicture}[scale=1.4, line cap=round]
	\draw (0,0) -- (0,1ex) -- (1ex,1ex) -- (1ex,0) -- (0,0);
	\fill (0ex,0.5ex) rectangle ++(0.5ex,0.5ex);
\end{tikzpicture}\kern-1pt
}
\DeclareMathOperator{\WEmid}{
	\begin{tikzpicture}[scale=1.4, line cap=round]
		\draw (0,0) -- (0,1ex) -- (1ex,1ex) -- (1ex,0) -- (0,0);
		\fill (0.25ex,0ex) rectangle ++(0.5ex,1ex);
	\end{tikzpicture}\kern-1pt
}

\DeclareMathOperator{\NSmid}{
	\begin{tikzpicture}[scale=1.4, line cap=round]
		\draw (0,0) -- (0,1ex) -- (1ex,1ex) -- (1ex,0) -- (0,0);
		\fill (0ex,0.25ex) rectangle ++(1ex,0.5ex);
	\end{tikzpicture}\kern-1pt
}

\DeclareMathOperator{\WE}{
\begin{tikzpicture}[scale=1.4, line cap=round]
	\draw (0,0) -- (0,1ex);%
	\draw (0,1ex) -- (1ex,1ex);%
	\draw (1ex,1ex) -- (1ex,0);%
	\draw (1ex,0) -- (0,0);%
	\draw (0.5ex,0) -- (0.5ex,1ex);%
\end{tikzpicture}\kern-1pt
}

\DeclareMathOperator{\NS}{
\begin{tikzpicture}[scale=1.4, line cap=round]
	\draw (0,0) -- (0,1ex);%
	\draw (0,1ex) -- (1ex,1ex);%
	\draw (1ex,1ex) -- (1ex,0);%
	\draw (1ex,0) -- (0,0);%
	\draw (0ex,0.5ex) -- (1ex,0.5ex);%
\end{tikzpicture}\kern-1pt
}

\tikzstyle{bvert}=[draw,circle,fill=black,minimum size=5pt,inner sep=0pt]
\tikzstyle{blvert}=[draw,circle,fill=blue,draw=blue,minimum size=5pt,inner sep=0pt]
\tikzstyle{wvert}=[draw,circle,fill=white,minimum size=5pt,inner sep=0pt]
\tikzstyle{blwvert}=[draw,circle,draw=blue,fill=white,minimum size=5pt,inner sep=0pt]
\tikzset{circle through 3 points/.style n args={3}{%
insert path={let    \p1=($(#1)!0.5!(#2)$),
                    \p2=($(#1)!0.5!(#3)$),
                    \p3=($(#1)!0.5!(#2)!1!-90:(#2)$),
                    \p4=($(#1)!0.5!(#3)!1!90:(#3)$),
                    \p5=(intersection of \p1--\p3 and \p2--\p4)
                    in
                 node at (\p5) [draw,circle through= {(#1)}]{}}
}}

\title{Discrete Kœnigs nets and finite Laplace sequences}
\author{Niklas Christoph Affolter
     \thanks{TU Wien, Institute of Discrete Mathematics and Geometry, Wiedener Hauptstr. 8-10/104, A-1040 Vienna, Austria.
       \textit{E-mail address}: \texttt{affolter@posteo.net}} \footnotemark[2]\ ,
      Alexander Yves Fairley
     \thanks{TU Berlin, Institute of Mathematics, Strasse des 17. Juni 136, 10623 Berlin, Germany.
	\textit{E-mail address}: \texttt{fairley@tu-berlin.de}} ,
}

\date{August 4, 2025} % this is to prevent the date from changing when arxiv recompiles

\begin{document}

\maketitle

\begin{abstract}
	Q-nets are maps from the square grid to projective space that have planar faces. 
	We consider the \emph{Laplace sequences} of Q-nets, which are determined by iterating a discrete time dynamics called \emph{Laplace transformations}. In general, the Laplace sequences are bi-infinite. However, there are special cases in which a Laplace transform degenerates to a curve. In these cases we say that the sequences terminates. In this paper, we consider two special cases of Q-nets which are both called \emph{(discrete) Kœnigs nets}. For these Kœnigs nets we show that if the sequence terminates, then the sequence is finite. More specifically, we show that if the Laplace transform is Laplace degenerate (or Goursat degenerate) after $m$ steps in one direction, then it is Laplace degenerate after $m+1$ (or $m+2$) steps in the other direction. 
\end{abstract}

\setcounter{tocdepth}{1}
\tableofcontents

\input{introduction}

\input{basics}

\input{results}

\input{projections}

\input{laplaceinvariants}

\input{terminatingsequences}

\input{symterminatingsequences}

\bibliographystyle{alpha}
\bibliography{references}

\end{document}

%% file: introduction.tex
\newpage

\section{Introduction}

Consider a \emph{Q-net}, which is a map $P\colon \Z^2 \rightarrow \RP^n$ such that the image of each unit-square is contained in a plane \cite{sauer1970differenzen, BS2008DDGbook}. Given a Q-net $P$ we define the \emph{(forwards) Laplace transform} \cite{doliwa1997geometricToda}
\begin{align}
	\mathcal{L}_{+}P(i,j) := (P(i,j)\vee P(i,j+1))\cap (P(i+1,j)\vee P(i+1,j+1)),
\end{align}
and the \emph{(backwards) Laplace transform}
\begin{align}
	\mathcal{L}_{-}P(i,j):= (P(i,j) \vee P(i+1,j))\cap (P(i,j+1) \vee P(i+1,j+1)).	
\end{align}
It turns out that both Laplace transforms are Q-nets. Thus, Laplace transformations can be iterated. In fact, $\mathcal{L}_+$ and $\mathcal{L}_-$ are mutually inverse transformations up to an index shift. Therefore, the iterated Laplace transforms $\mathcal{L}_{+}^{m}P$ and $\mathcal{L}_{-}^{m}P$ constitute a sequence called the \emph{Laplace sequence} of $P$ \cite{doliwa1997geometricToda}.
However, there are special cases when $\mathcal{L}_{+}^{m}P$ degenerates to a discrete curve for some $m \in \N$. In this case, we say that the Laplace sequence \emph{terminates} because $\mathcal{L}_{+}^{m+1}P$ is not well-defined. Sometimes the Laplace sequence terminates in both directions, in which case the sequence is \emph{finite}.

In this paper, we study terminating Laplace sequences of a special class of Q-nets, called \emph{(discrete) Kœnigs nets} \cite{BS2009Koenigsnets}.
Let us give an elementary characterization of Kœnigs nets. For this, note that each quad of a Q-net has a diagonal intersection point. 
The \emph{diagonal intersection net} of a Q-net consists of the diagonal intersection points. A Kœnigs net is a Q-net  such that the diagonal intersection net is also a Q-net. There are other algebraic and geometric characterizations of Kœnigs nets, see \cite{BS2008DDGbook}. Since Kœnigs nets are special cases of Q-nets, we may study the special properties of the Laplace sequences of Kœnigs nets. In particular, we prove the following new result.
\begin{statement}
	If the Laplace sequence of a Kœnigs net terminates, then the sequence is finite.
\end{statement}
Additionally, we are able to quantify this statement: if the Laplace sequence terminates after $m$ steps in one direction, then it terminates after at most $m+2$ steps in the other direction. We also prove the statement for a second type of Kœnigs nets, which was introduced in \cite{doliwa2003}. To distinguish the two types, we will refer to the first type as \emph{BS-Kœnigs net} and to the second type as \emph{D-Kœnigs net} in the bulk of the paper. Essentially, a Q-net is a D-Kœnigs net if it is the diagonal intersection net of a Q-net.

There are only two other cases where it is known that a terminating Laplace sequence must be finite. The first case deals with Q-nets in quadrics \cite{bobenkofairley2023circularnets}, which also covers circular Q-nets. In fact, we obtained our results by investigating local lifts of Kœnigs nets to degenerate quadrics and then adapting the techniques in \cite{bobenkofairley2023circularnets}. The other case is the \emph{pentagram map} \cite{schwartzaxisaligned} and related maps \cite{gdevron, yao, adtmdskpgeometry}. The common property of these maps is that they can be viewed as Laplace sequences of Q-nets with additional \emph{periodicity constraints} on the Q-net.

We believe our results are interesting purely from the viewpoint of discrete geometry. However, there is additional motivation coming from \emph{discrete differential geometry}, which is a branch of pure mathematics that studies structure-preserving discretizations of differential geometry, see \cite{BS2008DDGbook} for an introduction. In particular, (discrete) Kœnigs nets were introduced as discretizations of smooth Kœnigs nets, and they have many analogous properties. Indeed, although our results are new in the discrete setting, analogous results are known in the smooth theory \cite{Tzitzeica1924geometrie, darboux1915lecons}. Thus, our results give further evidence that the two discretizations of Kœnigs nets are appropriate - since they also behave analogously to smooth Kœnigs nets with respect to their Laplace sequences. 

That said, there is an interesting difference between the discrete and the smooth theory: if the Laplace sequence of a discrete Kœnigs net terminates in one direction, then in the other direction it generally terminates one step later than in the smooth theory. We propose two ways to resolve this difference. First, we prove the existence of special discrete Kœnigs nets for which terminations occurs after the same number of steps as in the smooth theory (and we show how to construct such Kœnigs nets). 
The second way is a change of perspective: we show that each discrete Kœnigs net belongs to a suitable pair of discrete Kœnigs nets (a special case of \emph{Kœnigs binets}, which are the subject of an upcoming publication \cite{adtbinets}). Viewed as pairs, the terminations occur after the same number of steps as in the smooth theory.

We believe it may be interesting to investigate how our results are affected if one considers the special case of \emph{discrete isothermic surfaces} \cite{bpdisosurfaces}, which are circular Kœnigs nets \cite{BS2008DDGbook}. Moreover, it is remarkable that there is a similar difference between the discrete and smooth theories for the Laplace sequences of circular nets with spherical parameter lines, see \cite{bobenkofairley2023circularnets}. Additionally, there has been renewed interest concerning (discrete and smooth) isothermic surfaces with spherical (or planar) parameter lines \cite{hoffmannSzewieczek2024isothermic, bobenko2023isothermic}. Further interesting examples of Q-nets with terminating Laplace sequences have recently appeared in \cite{KMT2023conenets, ppscone}. Another article by the authors concerning Kœnigs nets with constrained parameter spaces and finite Laplace sequences is to appear soon \cite{afconstrained}.

\subsection*{Plan of the paper}

In Section~\ref{sec:basics} we give a short review of the necessary fundamentals of projective geometry. In Section~\ref{sec:results} we give a summary of the main definitions and results. In Section~\ref{section:liftsofQnets} we develop a technique to lift Q-nets into maximal dimensions. This technique is crucial for the later proofs of the main results. The technique also allows us to make genericity assumptions precise. In Section~\ref{sec:laplaceseq} we collect general statements about Laplace invariants and Laplace sequences. In Section~\ref{sec:terminating} we prove our main results concerning finite Laplace sequences. In Section~\ref{sec:symterminating} we show how to construct Kœnigs nets with Laplace sequences that terminate after the same number of steps in both directions. Finally, in Section~\ref{sec:outlook} we give a few concluding remarks connecting our results to the literature. 
\subsection*{Acknowledgments}

N.~C.~Affolter and A.~Y.~Fairley were supported by the Deutsche Forschungsgemeinschaft (DFG) Collaborative Research Center TRR 109 ``Discretization in Geometry and Dynamics''. We would like to thank Christian Müller and Jan Techter for invaluable discussions.

%% file: basics.tex
\section{Geometry of $\RP^n$}\label{sec:basics}

We briefly review some basic notions of projective geometry necessary for the understanding of this paper, see also \cite{casas2014} for more background. 

For $x,y \in \R^{n+1} \setminus \{0\}$ we use the equivalence relation $x \sim y$ if and only if $x= \lambda y$ for some non-zero $\lambda \in \R$. We define \emph{projective space} $\RP^n$ as
\begin{align}
	\RP^n = (\R^{n+1} \setminus \{0\})\ /\ {\sim} = \{ [x] \ | \ x \in \R^{n+1} \setminus \{0\} \}.
\end{align}
For a point $[x] \in \RP^n$ we call $x \in \R^{n+1}\setminus \{0\}$ a \emph{representative vector} of $[x]$. Every $d$-dimensional projective subspace $A \subset \RP^n$ is the projectivization of a $(d+1)$-dimensional linear subspace $a \subset \R^{n+1}$, that is
\begin{align}
	A = [a] = \{[x] \mid x \in a \setminus \{0\}\}.
\end{align}
For two projective subspaces $A = [a]$, $B= [b]$ the \emph{join} $A \vee B$ is
\begin{align}
	A \vee B = [ \mbox{span}\{a,b\}].
\end{align}
Two projective subspaces $A, B \subset \RP^n$ are called \emph{supplementary} if
\begin{align}
	A \cap B = \emptyset \quad \mbox{and} \quad A \vee B = \RP^n.
\end{align}
Given two supplementary subspaces  $E, C \subset \RP^n$ the \emph{(central) projection} with center $C$ onto $E$ is the map 
\begin{align}
	\pi\colon \RP^n \setminus C \to E, \quad P \mapsto (P \vee C) \cap E.
\end{align}
For four points $P_1,P_2,P_3,P_4 \in \RP^1$ given by representative vectors $p_1,p_2,p_3,p_4\in \R^2$, the \emph{cross-ratio} is given by
\begin{align}
	\cro(P_1,P_2,P_3,P_4) = \frac{\det(p_1, p_2) \det(p_3, p_4)}{\det(p_2, p_3) \det(p_4,p_1)}.
\end{align}
The cross-ratio is independent of the choice of representative vectors, invariant under projective transformations, and invariant under central projections. For four points on a line $\RP^1 \subset \RP^n$ the cross-ratio is defined analogously by restricting to $\RP^1$.

Given a (non-trivial) symmetric bilinear form $\p \colon \R^{n+1} \times \R^{n+1} \to \R$, the associated \emph{quadric} is the set
\begin{align}
	\mathcal{Q} := \{[x] \in \RP^n \mid \p(x,x) = 0\}.
\end{align}
Non-zero scalar multiples of $\p$ define the same quadric $\mathcal Q$. Two points $[x], [y] \in \RP^n$ are \emph{conjugate} relative to $\mathcal{Q}$ if $\p(x,y) = 0$. The \emph{polar} of a point $P = [p] \in \RP^n$ with respect to $\mathcal{Q}$ is
\begin{align}
	P^\perp := \{[x] \in \RP^n \mid \p (p,x) = 0\}.
\end{align}
A point $P$ in $\mathcal{Q}$ is \emph{singular} if $P^\perp = \RP^n$. A non-empty quadric is \emph{non-degenerate} if it has no singular points. A projective subspace $A \subset \RP^n$ is an \emph{isotropic} subspace of a quadric $\mathcal{Q}$ if $A \subset \mathcal{Q}$.

%% file: results.tex
\section{Definitions and results}\label{sec:results}

Throughout, we let $m \in \N$ be a positive integer and for $a,b \in \N$ we write
\begin{align}
	\Sigma_m := \{0,1, \ldots, m\}, \quad  \Sigma_{a,b} := \Sigma_a \times \Sigma_b.
\end{align}
We write $\Sigma$ when a statement makes sense for both the finite and the infinite case, that is $\Sigma$ represents $\Sigma_{a,b}$ or $\Z^2$. Thus, Q-nets are defined with the domain $\Sigma$.

\begin{definition}
	A \emph{Q-net} is a map $\Sigma \rightarrow \RP^n$ such that the image of each face is contained in a plane.
\end{definition}
We are mostly interested in generic Q-nets in the sense of the following definition.
\begin{definition}\label{def:nondegqnet}
	A \emph{non-degenerate} Q-net $P: \Sigma \rightarrow \RP^n$ is a Q-net such that
	\begin{enumerate}
		\item for every edge the two corresponding points do not coincide, and
		\item for any three vertices of any face the corresponding points of $P$ span a plane. 
	\end{enumerate}
\end{definition}
Note that the second condition implies the first condition unless the domain of $P$ is $\Sigma_{a,0}$ or $\Sigma_{0,b}$. Next, we are interested in a type of discrete-time dynamics for Q-nets, given by iterating \emph{Laplace transformations}.
Laplace transformations of Q-nets were introduced by Doliwa \cite{doliwa1997geometricToda} as follows.
\begin{figure}
	\centering
	\begin{tikzpicture}[line cap=round,line join=round,>=triangle 45,x=1.0cm,y=1.0cm,scale=0.8]
		\definecolor{qqwuqq}{rgb}{0.,0.39215686274509803,0.}
		\clip(-2.06,-2.3) rectangle (12.42,5.56);
%		\draw (current bounding box.north east) -- (current bounding box.north west) -- (current bounding box.south west) -- (current bounding box.south east) -- cycle;
		\draw [line width=1pt,densely dashed] (5.64,2.46)-- (-1.62,-1.);
		\draw [line width=1pt,densely dashed] (-1.62,-1.)-- (5.5,5.16);
		\draw [line width=1pt,densely dashed] (5.5,5.16)-- (5.64,2.46);
		\draw [line width=1pt,densely dashed] (5.64,2.46)-- (12.16,-1.48);
		\draw [line width=1pt,densely dashed] (12.16,-1.48)-- (5.5,5.16);
		\draw [line width=1pt,densely dashed] (5.64,2.46)-- (5.3,0.74);
		\draw [line width=1pt,densely dashed] (5.3,0.74)-- (0.5278102388403072,0.023612042202129757);
		\draw [line width=1pt,densely dashed] (5.3,0.74)-- (10.566791115099843,-0.5172326677137082);
		\draw [line width=1pt,densely dashed] (3.749048234236283,3.6451316183842)-- (3.905296931048227,1.6332682343563176);
		\draw [line width=1pt,densely dashed] (3.905296931048227,1.6332682343563176)-- (3.8914355385497297,0.5285501900662476);
		\draw [line width=1pt,densely dashed] (7.633390140055851,0.182997786225348)-- (7.151434597262734,1.546648418218532);
		\draw [line width=1pt,densely dashed] (7.151434597262734,1.546648418218532)-- (7.819811895183212,2.847154506904426);
		\draw [line width=1.5pt,color=qqwuqq] (0.5278102388403072,0.023612042202129757)-- (-1.62,-1.);
		\draw [line width=1.5pt,color=qqwuqq] (-1.62,-1.)-- (12.16,-1.48);
		\draw [line width=1.5pt,color=qqwuqq] (12.16,-1.48)-- (10.566791115099843,-0.5172326677137082);
		\draw [line width=1.5pt,color=qqwuqq] (10.566791115099843,-0.5172326677137082)-- (0.5278102388403072,0.023612042202129757);
		\draw [line width=1.5pt] (3.749048234236283,3.6451316183842)-- (5.5,5.16);
		\draw [line width=1.5pt] (5.5,5.16)-- (5.64,2.46);
		\draw [line width=1.5pt] (5.64,2.46)-- (3.905296931048227,1.6332682343563176);
		\draw [line width=1.5pt] (3.905296931048227,1.6332682343563176)-- (3.749048234236283,3.6451316183842);
		\draw [line width=1.5pt] (3.905296931048227,1.6332682343563176)-- (3.8914355385497297,0.5285501900662476);
		\draw [line width=1.5pt] (3.8914355385497297,0.5285501900662476)-- (5.3,0.74);
		\draw [line width=1.5pt] (5.3,0.74)-- (5.64,2.46);
		\draw [line width=1.5pt] (5.3,0.74)-- (7.633390140055851,0.182997786225348);
		\draw [line width=1.5pt] (7.633390140055851,0.182997786225348)-- (7.151434597262734,1.546648418218532);
		\draw [line width=1.5pt] (7.151434597262734,1.546648418218532)-- (5.64,2.46);
		\draw [line width=1.5pt] (5.5,5.16)-- (7.819811895183212,2.847154506904426);
		\draw [line width=1.5pt] (7.819811895183212,2.847154506904426)-- (7.151434597262734,1.546648418218532);
		\begin{small}
			\draw [fill=black] (5.64,2.46) circle (2.5pt);
			\draw[color=black] (6.4,2.73) node {$P(1,1)$};
			\draw [fill=qqwuqq] (-1.62,-1.) circle (2.5pt);
			\draw[color=qqwuqq] (-1.0,-1.4) node {$P_{-1}(0,1)$};
			\draw [fill=black] (5.5,5.16) circle (2.5pt);
			\draw[color=black] (6.5,5.17) node {$P(1,2)$};
			\draw [fill=qqwuqq] (12.16,-1.48) circle (2.5pt);
			\draw[color=qqwuqq] (11.29,-1.76) node {$P_{-1}(1,1)$};
			\draw [fill=black] (5.3,0.74) circle (2.5pt);
			\draw[color=black] (6.2,0.9) node {$P(1,0)$};
			\draw [fill=qqwuqq] (0.5278102388403072,0.023612042202129757) circle (2.5pt);
			\draw[color=qqwuqq] (1.0,-0.4) node {$P_{-1}(0,0)$};
			\draw [fill=qqwuqq] (10.566791115099843,-0.5172326677137082) circle (2.5pt);
			\draw[color=qqwuqq] (9.67,-0.8) node {$P_{-1}(1,0)$};
			\draw [fill=black] (3.749048234236283,3.6451316183842) circle (2.5pt);
			\draw[color=black] (2.9,3.83) node {$P(0,2)$};
			\draw [fill=black] (3.905296931048227,1.6332682343563176) circle (2.5pt);
			\draw[color=black] (3,1.83) node {$P(0,1)$};
			\draw [fill=black] (3.8914355385497297,0.5285501900662476) circle (2.5pt);
			\draw[color=black] (3.9,0.1) node {$P(0,0)$};
			\draw [fill=black] (7.633390140055851,0.182997786225348) circle (2.5pt);
			\draw[color=black] (6.8,-0.05) node {$P(2,0)$};
			\draw [fill=black] (7.151434597262734,1.546648418218532) circle (2.5pt);
			\draw[color=black] (8,1.69) node {$P(2,1)$};
			\draw [fill=black] (7.819811895183212,2.847154506904426) circle (2.5pt);
			\draw[color=black] (8.4,3.21) node {$P(2,2)$};
		\end{small}
	\end{tikzpicture}
	\caption{The Laplace transform $\mathcal L_-P = P_{-1}$ (green) of a Q-net $P$ (black).}
	\label{fig:laplacetransform}
\end{figure}
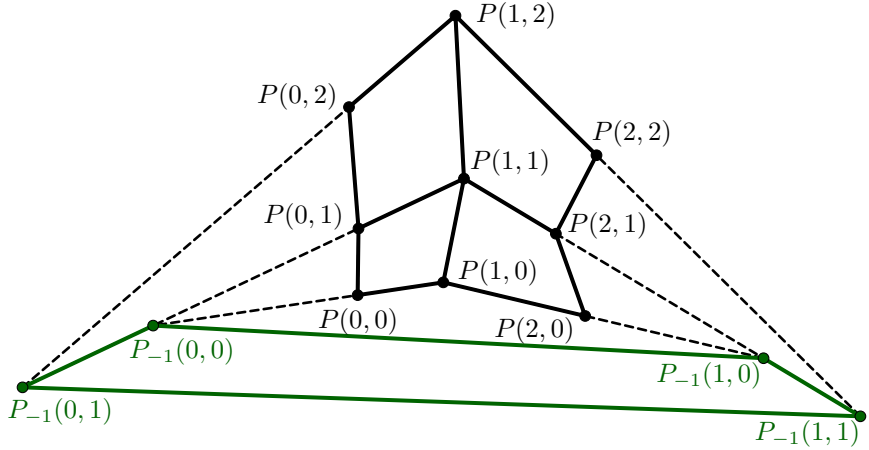
\begin{definition}\label{def:laplace}
	For a non-degenerate Q-net $P\colon \Sigma \rightarrow \RP^n$, the \emph{Laplace transforms} $\mathcal{L}_{+}P$ and $\mathcal{L}_{-}P$ are defined by 
	\begin{align}
		\mathcal{L}_{+}P(i,j) := (P(i,j)\vee P(i,j+1))\cap (P(i+1,j)\vee P(i+1,j+1)), \label{eq:laplaceforward}\\
		\mathcal{L}_{-}P(i,j):= (P(i,j) \vee P(i+1,j))\cap (P(i,j+1) \vee P(i+1,j+1)),\label{eq:laplacebackward}
	\end{align}
    see also Figure~\ref{fig:laplacetransform}.
\end{definition}
The intersection points in Definition~\ref{def:laplace} exist since the corresponding two lines are contained in the plane of the corresponding quad of the Q-net $P$, see Figure~\ref{fig:laplacetransform}. They are well-defined since $P$ is non-degenerate.  Note that for a Q-net $P$ defined on a finite patch $\Sigma_{a,b}$ the Laplace transforms $\mathcal{L}_\pm P$ are defined on the smaller patch $\Sigma_{a-1,b-1}$.

It is straightforward to check that the Laplace transforms $\mathcal{L}_+ P$ and $\mathcal{L}_- P$ are Q-nets, see for example \cite[Exericse 2.13]{BS2008DDGbook}.
Assuming that the Laplace transforms are non-degenerate, Laplace transformations can be iterated. To lighten the notation, for $m \geq 0$ we write
\begin{align}
	P_m := (\mathcal L_+)^m P, \quad \mbox{and} \quad P_{-m} := (\mathcal L_-)^{m}P,
\end{align}
for the $m$-times iterated Laplace transforms. 

It readily follows from Definition~\ref{def:laplace} that the Laplace transformations $\mathcal{L}_+$ and $\mathcal{L}_-$ are mutually inverse up to an index shift, that is:

\begin{align}
    \mathcal{L}_+ \circ \mathcal{L}_- P (i,j) = \mathcal{L}_- \circ \mathcal{L}_+ P (i,j) = P(i+1,j+1).
\end{align}

The iterated Laplace transforms determine the {\em Laplace sequence}
\begin{align}
    \ldots  \leftarrow P_{-3} \leftarrow  P_{-2}  \leftarrow P_{-1}  \leftarrow P \rightarrow P_{1} \rightarrow P_{2} \rightarrow P_{3} \rightarrow \ldots
\end{align}

Generically, if $P$ is defined on $\Z^2$, the Laplace sequence is infinite in both directions. If an iterated Laplace transform is degenerate, then the Laplace transformation (as defined in Definition~\ref{def:laplace}) is not defined anymore and we say that the Laplace sequence \emph{terminates}. Note that in the present paper our main theorems only deal with cases of terminating Laplace sequences, as we explain below.
\begin{remark}
    Note that the non-degeneracy condition for a Q-net $P$ guarantees that both $\mathcal L_+ P$ and $\mathcal L_- P$ are well-defined. However, there are degenerate cases where $\mathcal L_+ P$ is well-defined but $\mathcal L_- P$ is not (or vice versa). These cases may also be interesting, but we do not consider them in this paper. 
\end{remark}

There are two types of degeneracy that we investigate in the paper, see Figure~\ref{fig:simpledegenerations} (see also \cite{fairley2023thesis}). The first type of degeneracy is the following.

\begin{figure}[tb]
	\centering
	\includegraphics[width=0.4\textwidth]{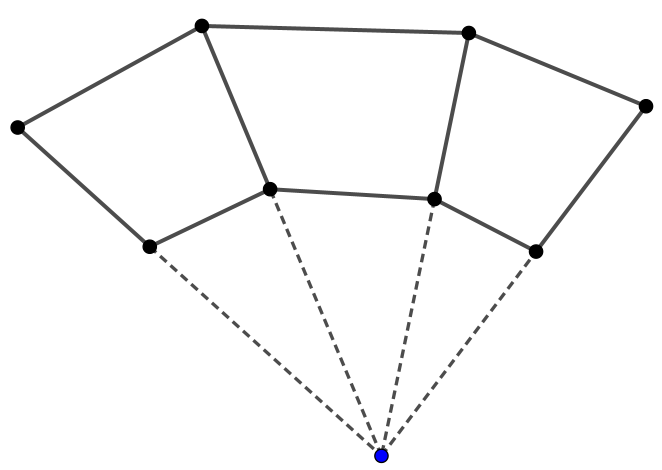}
	\hspace{10mm}
	\includegraphics[width=0.4\textwidth]{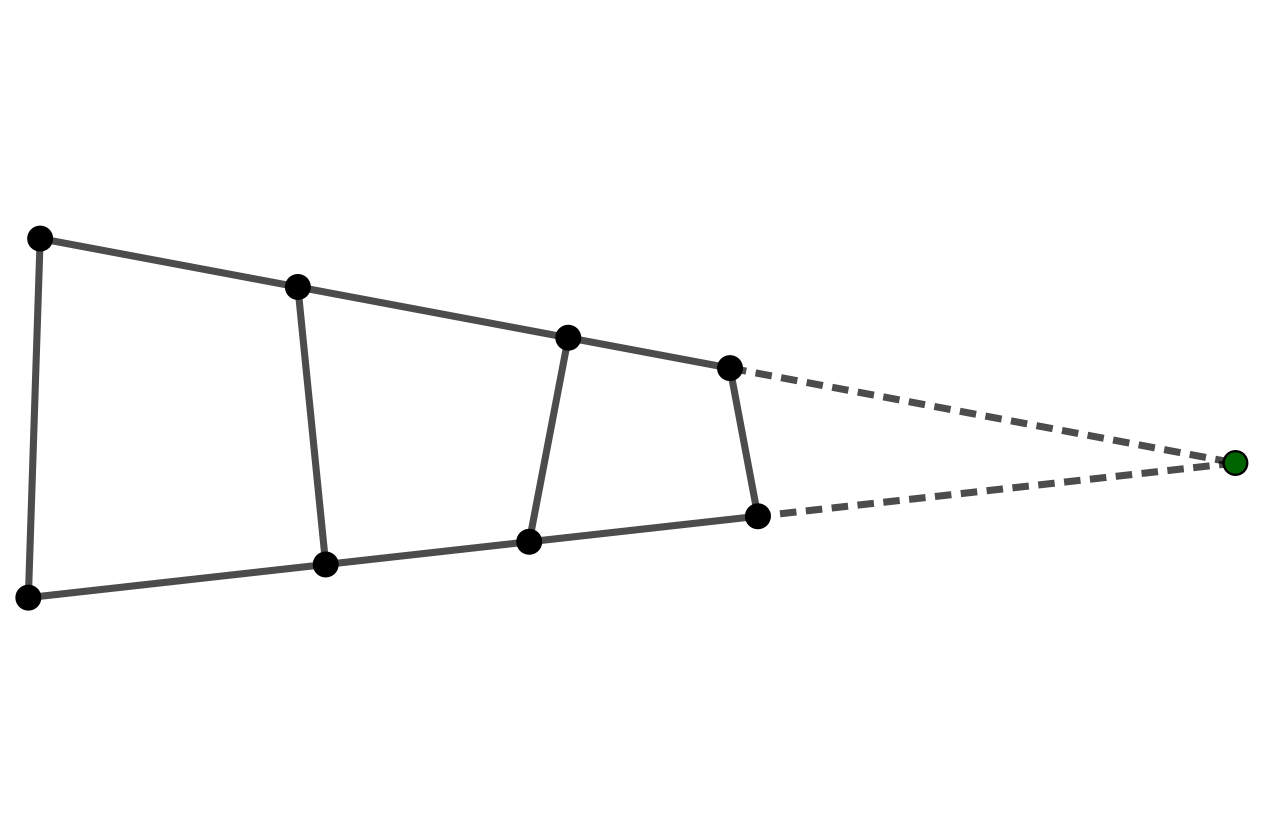}
	\caption{Left: a Q-net $P$ (black) such that $P_1$ (blue point) is Laplace degenerate. Right: a Q-net $P$ (black) such that $P_{-1}$ (green point) is Goursat degenerate.}
	\label{fig:simpledegenerations}
\end{figure}

\begin{definition}\label{defn:laplace}
	Let $P\colon \Sigma \to \RP^n$ be a Q-net, then
	\begin{enumerate}
		\item $P_{m}$ is \emph{Laplace degenerate} if $P_{m} (i,j)$  is independent of $i$ for all $j$,
		\item $P_{-m}$ is \emph{Laplace degenerate} if $P_{-m} (i,j)$  is independent of $j$ for all $i$.
	\end{enumerate}
\end{definition}
For Q-nets that are not Laplace generate, we introduce the following genericity condition.
\begin{definition}\label{defn:nowherelaplace}
	Let $P\colon \Sigma \to \RP^n$ be a non-degenerate Q-net.
	\begin{enumerate}
		\item $P_m$ is \emph{nowhere Laplace degenerate} if $P_m(i,j) \neq P_m(i+1,j)$ for all $i,j$.
		\item $P_{-m}$ is \emph{nowhere Laplace degenerate} if $P_{-m}(i,j) \neq P_{-m}(i,j+1)$ for all $i,j$.
	\end{enumerate}
\end{definition}
We make use of this genericity condition in the second type of degeneracy, which is defined as follows.
\begin{definition}\label{defn:goursat}
	Let $P\colon \Sigma \to \RP^n$ be a Q-net, then
	\begin{enumerate}	
		\item $P_{m}$ is \emph{Goursat degenerate} if $P_{m} (i,j)$  is independent of $j$ for all $i$, and $P_m$ is nowhere Laplace degenerate, 
		\item $P_{-m}$ is \emph{Goursat degenerate} if $P_{-m} (i,j)$ is independent of $i$ for all $j$,  and $P_{-m}$ is nowhere Laplace degenerate.
	\end{enumerate}
\end{definition}
If we know that $P$ is a Q-net in some Laplace sequence such that  $P(i,j)$ is independent of $i$ for all $j$, that is \emph{not} enough to distinguish whether $P$ is Laplace degenerate or Goursat degenerate. To determine this we also need to know whether $P$ was reached by forwards or backwards Laplace transforms.
Also note that if $P_{m}$ is Laplace or Goursat degenerate, that means that $P_m$ is a discrete curve -- either $(P_m(i,j))_j$ or $(P_m(i,j))_i$. 

Moreover, we will see that we may treat Goursat degeneracy as a special case of Laplace degeneracy by considering lifts of Goursat degenerate Q-nets to a projective space with higher dimension. These lifts turn out to be non-degenerate Q-nets. However, the lifts are Laplace degenerate after one more step of Laplace transformations (Lemma~\ref{lem:goursatliftstolaplace}).

\begin{remark}\label{rem:mixedtype}
	Note that we included a genericity condition in the definition of Goursat degeneracy. Without the genericity constraint, one could consider Q-nets that are both Laplace \mbox{and} Goursat degenerate, sometimes called \emph{degenerate of mixed type}. However, as we will see later, Laplace degeneracy dominates Goursat degeneracy. Thus, our exposition is simplified by excluding the Laplace degenerate case from the Goursat degenerate case in the definition, instead of the lemmas and theorems.
\end{remark}

Next, we introduce the Laplace invariants \cite{doliwa1997geometricToda}, which are projective invariants of a Q-net, and obey rational recursion equations (see Section~\ref{sec:laplaceseq}).

\begin{definition}\label{def:discreteLaplaceinvariants}
	Let $P\colon \Z^2  \to \RP^n$ be a non-degenerate Q-net. The \emph{Laplace invariants} $H\colon \Z^2 \to \R$ and $K\colon \Z^2 \to \R$ are defined as
	\begin{align}
	       	H(i,j)&:= \cro(P(i,j), P_{1}(i,j), P(i,j+1), P_{1}(i-1,j)), \\
		   K(i,j) &:= \cro(P(i,j), P_{-1}(i,j) ,P(i+1,j) , P_{-1}(i,j-1)).
	\end{align}
\end{definition}
For both cross-ratios in Definition~\ref{def:discreteLaplaceinvariants}, the four points in the argument are on a line as required, see also Figure~\ref{fig:laplacetransform}. Moreover, combinatorially we think of $H(i,j)$ as being assigned to the vertical edge joining $(i,j)$ and $(i,j+1)$, whereas $K(i,j)$ is assigned to the horizontal edge joining $(i,j)$ and $(i+1,j)$ -- see also Figure~\ref{fig:combinatorialinvariants}.

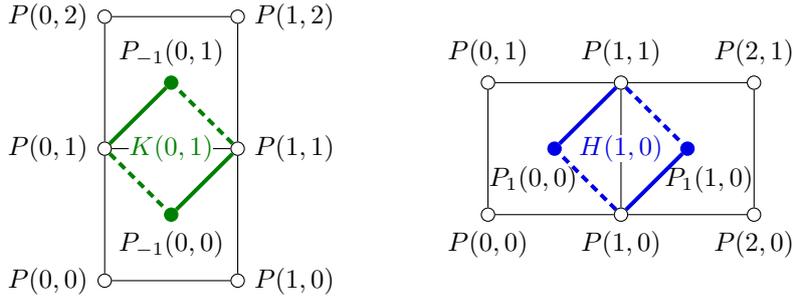
\begin{figure}[tb]
	\centering
	\begin{tikzpicture}[scale=1.75,baseline=(current bounding box.center)]
		\small
		\node[wvert, label={left:$P(0,0)$}] (p00) at (0,0) {};
		\node[wvert, label={left:$P(0,1)$}] (p01) at (0,1) {};
		\node[wvert, label={left:$P(0,2)$}] (p02) at (0,2) {};
		\node[wvert, label={right:$P(1,0)$}] (p10) at (1,0) {};
		\node[wvert, label={right:$P(1,1)$}] (p11) at (1,1) {};
		\node[wvert, label={right:$P(1,2)$}] (p12) at (1,2) {};
		\node[bvert, black!50!green, label={below:$P_{-1}(0,0)$}] (f00) at (0.5,0.5) {};
		\node[bvert, black!50!green, label={above:$P_{-1}(0,1)$}] (f01) at (0.5,1.5) {};			
		\draw[-]
			(p00) -- (p01) -- (p02) -- (p12) -- (p11) -- (p10) -- (p00)
			(p01) edge[] node [fill=white, inner sep=0, text=black!50!green] {$K(0,1)$}  (p11)
		;
		\draw[-, line width=1.5, black!50!green]
			(p01) edge[-] (f01) edge[densely dashed] (f00)
			(p11) edge[-] (f00) edge[densely dashed] (f01)
		;
	\end{tikzpicture}
	\hspace{10mm} 
	\begin{tikzpicture}[scale=1.75,baseline=(current bounding box.center)]
		\small
		\node[wvert, label={below:$P(0,0)$}] (p00) at (0,0) {};
		\node[wvert, label={above:$P(0,1)$}] (p01) at (0,1) {};
		\node[wvert, label={below:$P(2,0)$}] (p20) at (2,0) {};
		\node[wvert, label={below:$P(1,0)$}] (p10) at (1,0) {};
		\node[wvert, label={above:$P(1,1)$}] (p11) at (1,1) {};
		\node[wvert, label={above:$P(2,1)$}] (p21) at (2,1) {};
		\node[bvert, black!10!blue, label={[xshift=14,yshift=-1]below left:$P_{1}(0,0)$}] (f00) at (0.5,0.5) {};
		\node[bvert, black!10!blue, label={[xshift=-14,yshift=-1]below right:$P_{1}(1,0)$}] (f10) at (1.5,0.5) {};			
		\draw[-]
			(p00) -- (p10) -- (p20) -- (p21) -- (p11) -- (p01) -- (p00)
			(p10) edge[] node [fill=white, inner sep=0, text=black!10!blue] {$H(1,0)$}  (p11)
		;
		\draw[-, line width=1.5, black!10!blue]
			(p10) edge[-] (f10) edge[densely dashed] (f00)
			(p11) edge[-] (f00) edge[densely dashed] (f10)
		;
	\end{tikzpicture}	
	\caption{Combinatorial picture of the Laplace invariants $H$ (left) and $K$ (right).}
	\label{fig:combinatorialinvariants}
\end{figure}

In this paper, we study the Laplace sequence for special types of Q-nets: \emph{discrete Kœnigs nets}, which are a  discretization of \emph{(smooth) Kœnigs nets} \cite{BS2008DDGbook}.
Kœnigs nets are studied in differential geometry for several reasons, maybe most importantly because if a surface has a curvature-line parametrization that is also a Kœnigs net, then the surface is an isothermic surface.
There are two definitions of discrete Kœnigs nets in the literature. The more established discretization in discrete differential geometry we call \emph{BS-Kœnigs nets}. These were introduced in \cite{BS2009Koenigsnets}. They have previously appeared in \cite{Doliwa2007} and \cite{sauer1933wackelige}, but without the interpretation as a discretization of Kœnigs nets.
\begin{definition}\label{defn:BSKoenigs}
	Let $P\colon \Sigma \to \RP^n$ be a non-degenerate Q-net. We call $P$ a {\em BS-Kœnigs net} if
	\begin{align}\label{eq:BScrossratioequ}
		H(i,j) \cdot H(i,j+1) = K(i,j+1) \cdot K(i-1,j+1).
	\end{align}
	for all $(i,j) \in \Sigma$.
\end{definition}
Note that by our definition a BS-Kœnigs net is always a non-degenerate Q-net. Also, see \cite{BS2009Koenigsnets} for several other characterizations of BS-Kœnigs nets. The second definition of discrete Kœnigs nets we call \emph{D-Kœnigs nets} \cite{doliwa2003}.

\begin{definition}\label{defn:DoliwaKoenigs}
	Let $D \colon \Sigma \to \RP^n$ be a non-degenerate Q-net. We call $D$ a {\em D-Kœnigs net} if 
	\begin{align}\label{eq:Dcrossratioequ}
		H(i,j) \cdot H(i+1,j) = K(i,j) \cdot K(i,j+1)
	\end{align} 
	for all $(i,j)\in \Sigma$.
\end{definition}
Both definitions of discrete Kœnigs nets involve the Laplace invariants of quadruples of edges, see also Figure~\ref{figure:twotypesofkoenigsnets}. 
Let us also note that there are other characterization of D-Kœnigs nets, for example the property that the six points
\begin{align}
	D_{1}(i-1,j),\ D_{1}(i,j),\ D_{1}(i+1,j),\ D_{-1}(i,j-1),\ D_{-1}(i,j),\ D_{-1}(i,j+1),
\end{align}
are contained in a conic. The equivalence to Definition~\ref{defn:DoliwaKoenigs} is shown in \cite[Exercise 2.21]{BS2008DDGbook}.

\begin{remark}
    In differential geometry, Kœnigs nets are characterized among the conjugate nets (the smooth analogue of Q-nets) by the condition that the smooth Laplace invariants $h, k$ satisfy $h = k$. That is, the horizontal and vertical Laplace invariants coincide. However, for discrete Kœnigs nets there are no horizontal and vertical Laplace invariants at the same edge, since an edge is either horizontal or vertical. Thus, it is not natural to require something like $H = K$ in the discrete theory. Instead, both discretizations are defined via an ``averaged'' version of $H=K$. 
\end{remark}

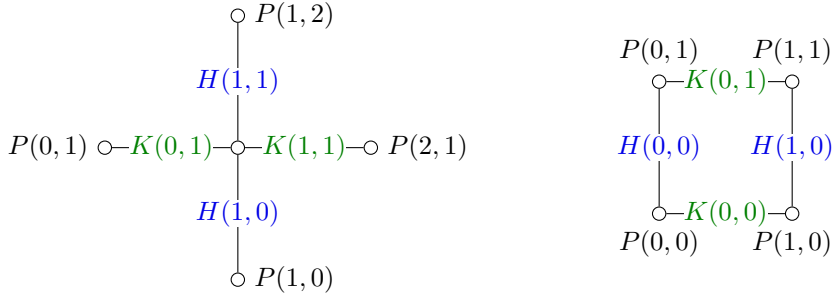
\begin{figure}[tb]
	\centering
	\begin{tikzpicture}[scale=1.75,baseline=(current bounding box.center)]
		\small
		\node[wvert, label={left:$P(0,1)$}] (p01) at (0,1) {};
		\node[wvert, label={right:$P(1,0)$}] (p10) at (1,0) {};
		\node[wvert] (p11) at (1,1) {}; %, label={right:$P(1,1)$}
		\node[wvert, label={right:$P(2,1)$}] (p21) at (2,1) {};
		\node[wvert, label={right:$P(1,2)$}] (p12) at (1,2) {};
		\draw[-]
			(p11) edge[] node [fill=white, inner sep=0, text=black!10!blue] {$H(1,0)$} (p10) 
			edge[] node [fill=white, inner sep=0, text=black!10!blue] {$H(1,1)$} (p12) 
			edge[] node [fill=white, inner sep=0, text=black!50!green] {$K(0,1)$} (p01) 
			edge[] node [fill=white, inner sep=0, text=black!50!green] {$K(1,1)$} (p21)
		;
	\end{tikzpicture}
	\hspace{15mm} 
	\begin{tikzpicture}[scale=1.75,baseline=(current bounding box.center)]
		\small
		\node[wvert, label={below:$P(0,0)$}] (p00) at (0,0) {};
		\node[wvert, label={above:$P(0,1)$}] (p01) at (0,1) {};
		\node[wvert, label={below:$P(1,0)$}] (p10) at (1,0) {};
		\node[wvert, label={above:$P(1,1)$}] (p11) at (1,1) {};
		\draw[-]
			(p00) edge[] node [fill=white, inner sep=0, text=black!50!green] {$K(0,0)$} (p10) edge[] node [fill=white, inner sep=0, text=black!10!blue] {$H(0,0)$} (p01)
			(p11) edge[] node [fill=white, inner sep=0, text=black!50!green] {$K(0,1)$} (p01) edge[] node [fill=white, inner sep=0, text=black!10!blue] {$H(1,0)$} (p10)
		;
	\end{tikzpicture}
	\hspace{5mm}
	\caption{If the product of the two $H$ equals the product of the two $K$ everywhere, the Q-net is a BS-Kœnigs net (left) or D-Kœnigs net (right) respectively.}
	\label{figure:twotypesofkoenigsnets}
\end{figure}

We have now all the ingredients to present our first main result, that if the Laplace sequence of a Kœnigs net terminates in one direction due to a Laplace or Goursat degeneration, then the sequence is finite.
\begin{theorem}\label{thm:BSandDKoenigslaplacedegunsymmetric}
	Let $P\colon \Sigma \to \RP^n$ be either a BS-Kœnigs net or a D-Kœnigs net. 
	\begin{enumerate}
		\item If $P_{m}$ is Laplace degenerate, then $P_{-m-1}$ is Laplace degenerate (assuming it exists). 
		\item If $P_{m}$ is Goursat degenerate, then $P_{-m-2}$ is Laplace degenerate (assuming it exists).
	\end{enumerate}
\end{theorem}
We develop a proof for this theorem in Section~\ref{sec:terminating}. Also, recall that we assume throughout that $m > 0$. However, it is clear that Theorem~\ref{thm:BSandDKoenigslaplacedegunsymmetric} can also be stated for $m < 0$ using symmetry. For example, if $m$ is negative and $P_{m}$ is Laplace degenerate, then $P_{-m+1}$ (and not $P_{-m-1}$) is also Laplace degenerate (assuming it exists). 

The existence assumptions in Theorem~\ref{thm:BSandDKoenigslaplacedegunsymmetric} are genericity assumptions. In particular, consider a Kœnigs net $P$ such that $P_m$ (and therefore $P_{-m-1})$ is Laplace degenerate. Moreover, define the map $Q$ via $Q(i,j) = P(j,i)$. Then, $Q_{m+1}$ is Laplace degenerate. If we now apply Theorem~\ref{thm:BSandDKoenigslaplacedegunsymmetric} to $Q$, it would appear that $Q_{-m-2}$ is Laplace degenerate. However, $Q$ does not satisfy the genericity assumption: $Q_{-m-2}$ does not exist. Note that if $Q_{-m-2}$ did exist, it would equal $P_{m+2}$ -- but $P_{m+2}$ does not exist since $P_{m}$ is Laplace degenerate.

Also, we will discuss how to construct BS-Kœnigs nets from boundary data, such that $P_{m}$ is Laplace degenerate and such that $P_{-m-1}$ \emph{does} exist (see Section~\ref{sec:symterminating}). We will also see that there are BS-Kœnigs nets such that $P_{m}$ is Laplace degenerate and such that $P_{-m-1}$ \emph{does not} exist. In fact, we show how to construct BS-Kœnigs nets, such that $P_{m}$ is Laplace degenerate and such that $P_{-m}$ is also Laplace degenerate, by putting additional constraints on the boundary data.

\begin{remark}\label{rem:tooearlyone}
    Another reason why we are interested in BS-Kœnigs nets such that $P_m$ and $P_{-m}$ are Laplace degenerate arises from the following observation.
    In differential geometry, if the Laplace sequence of a (smooth) Kœnigs net is Laplace degenerate after $m$ steps, then the Laplace sequence is (generally) also Laplace degenerate after $m$ steps in the other direction \cite{Tzitzeica1924geometrie}. Thus, it is surprising that the termination in Theorem~\ref{thm:BSandDKoenigslaplacedegunsymmetric} (generally) occurs one step later than in the smooth theory. 
    This difference leads to the natural question: are there discrete Kœnigs nets such that both $P_m$ and $P_{-m}$ are Laplace degenerate? This, we answer in the positive in Section~\ref{sec:symterminating}, where we also explain how to construct these discrete Kœnigs nets from boundary data.
\end{remark}

There is a way to make Theorem~\ref{thm:BSandDKoenigslaplacedegunsymmetric} more symmetric, by considering the interplay between BS-Kœnigs nets and D-Kœnigs nets, as we explain in the following.

\begin{definition}
	Let $P\colon \Sigma_{a,b} \to \RP^n$ be a non-degenerate Q-net. The \emph{diagonal intersection net $D$ of $P$} is the map
	\begin{align}
		D\colon  \Sigma_{a-1,b-1} \rightarrow \RP^n, \quad (i,j) \mapsto \big(P(i,j) \vee P(i+1,j+1)\big) \cap \big(P(i+1,j)\vee P(i,j+1)\big).
	\end{align}	
\end{definition}
Diagonal intersection nets relate BS-Kœnigs and D-Kœnigs nets as follows.
\begin{lemma}\label{lem:BSKoenigsandDiagonalIntersectionNetisDKoenigs}
	Let $P$ be a BS-Kœnigs net. Then, the diagonal intersection net $D$ is a Q-net, and if $D$ is non-degenerate then $D$ is a D-Kœnigs net.
\end{lemma}
\proof{
    The first part of Lemma~\ref{lem:BSKoenigsandDiagonalIntersectionNetisDKoenigs} is \cite[Theorem~2.26]{BS2008DDGbook} and the second follows from \cite[Exercise~2.21]{BS2008DDGbook} and Definition~\ref{defn:DoliwaKoenigs}. \qed
}

Thus, every BS-Kœnigs net $P$ has a D-Kœnigs net $D$ attached to it in a very natural way. It will sometimes be beneficial to think of the pair $P,D$ as being a joint description of one discrete surface.

\begin{remark}
    In Definition~\ref{defn:extensive} we will introduce the notion of an \emph{extensive Q-net}. For extensive Q-nets Lemma~\ref{lem:BSKoenigsandDiagonalIntersectionNetisDKoenigs} can be strengthened to state that $D$ is a Q-net if and only if $P$ is a BS-Kœnigs net, and then $D$ is non-degenerate and a D-Kœnigs net. 
\end{remark}

Let us denote by $D_m$ the $m$-th Laplace transform of $D$ -- not the diagonal intersection net of $P_m$. With that in mind we make the following observation.
\begin{theorem}\label{thm:symmetryLaplacedegenerateBSandD}
	Let $P$ be a BS-Kœnigs net and let $D$ be its diagonal intersection net.
	\begin{enumerate}
		\item $P_{m}$ is Laplace degenerate if and only if $D_{-m}$ is Laplace degenerate (assuming $P_m$ and $D_{-m}$ exist).
		\item If $P_{m}$ is Goursat degenerate, then $D_{-m-1}$ is Laplace degenerate (assuming $D_{-m-1}$ exists).
	\end{enumerate}
\end{theorem}
\proof{
    Follows from Proposition~\ref{prop:bsdlaplace} and Proposition~\ref{prop:bsdgoursat}.\qed
}

Recall that Theorem~\ref{thm:BSandDKoenigslaplacedegunsymmetric} states that if $P_m$ is Laplace degenerate then generically $P_{-m-1}$ is Laplace degenerate. Assuming the existence of all maps involved, we may combine Theorem~\ref{thm:BSandDKoenigslaplacedegunsymmetric} and Theorem~\ref{thm:symmetryLaplacedegenerateBSandD} as follows:
\begin{align}
	P_m \mbox{ Laplace degenerate} \quad &\xRightarrow{\mbox{\small Theorem }\ref{thm:BSandDKoenigslaplacedegunsymmetric}} \quad P_{-m-1} \mbox{ Laplace degenerate}, \\
	P_m \mbox{ Laplace degenerate} \quad &\xRightarrow{\mbox{\small Theorem }\ref{thm:symmetryLaplacedegenerateBSandD}} \quad D_{-m} \mbox{ Laplace degenerate}, \\
	D_{-m} \mbox{ Laplace degenerate} \quad &\xRightarrow{\mbox{\small Theorem }\ref{thm:BSandDKoenigslaplacedegunsymmetric}} \quad D_{m+1} \mbox{ Laplace degenerate}.
\end{align}
Thus, if $P_m$ is Laplace degenerate we conclude that (generically) all four maps $P_m$, $D_{m+1}$, $P_{-m-1}$ and $D_{-m}$ are Laplace degenerate. In particular, this means that for (otherwise generic) Laplace degenerate $P_m$ it is $D_{m+1}$ that is Laplace degenerate -- that is the forward singularity occurs one step \emph{later} in $D$ compared to $P$. In this case, the (non-generic) $P_{-m-1}$ is Laplace degenerate as well, but it is $D_{-m}$ that is Laplace degenerate -- that is the backward singularity occurs one step \emph{earlier} in $D$ compared to $P$.

\begin{remark}
    As mentioned in Remark~\ref{rem:tooearlyone}, the Laplace degeneracy in Theorem~\ref{thm:BSandDKoenigslaplacedegunsymmetric} occurs one step later than it does for smooth Kœnigs nets. However, by viewing $(P_m,D_{m+1})$ as a degenerate forward pair and $(D_{-m},P_{-m-1})$ as a degenerate backward pair the behaviour looks more symmetric.
\end{remark}

%% file: projections.tex
\section{Lifts of Q-nets}\label{section:liftsofQnets}

In this section we develop a technique of lifting Q-nets to higher dimensions, which will prove helpful on several occasions in the remainder of the paper.
\begin{figure}
	\centering
	\includegraphics[width=.6\textwidth]{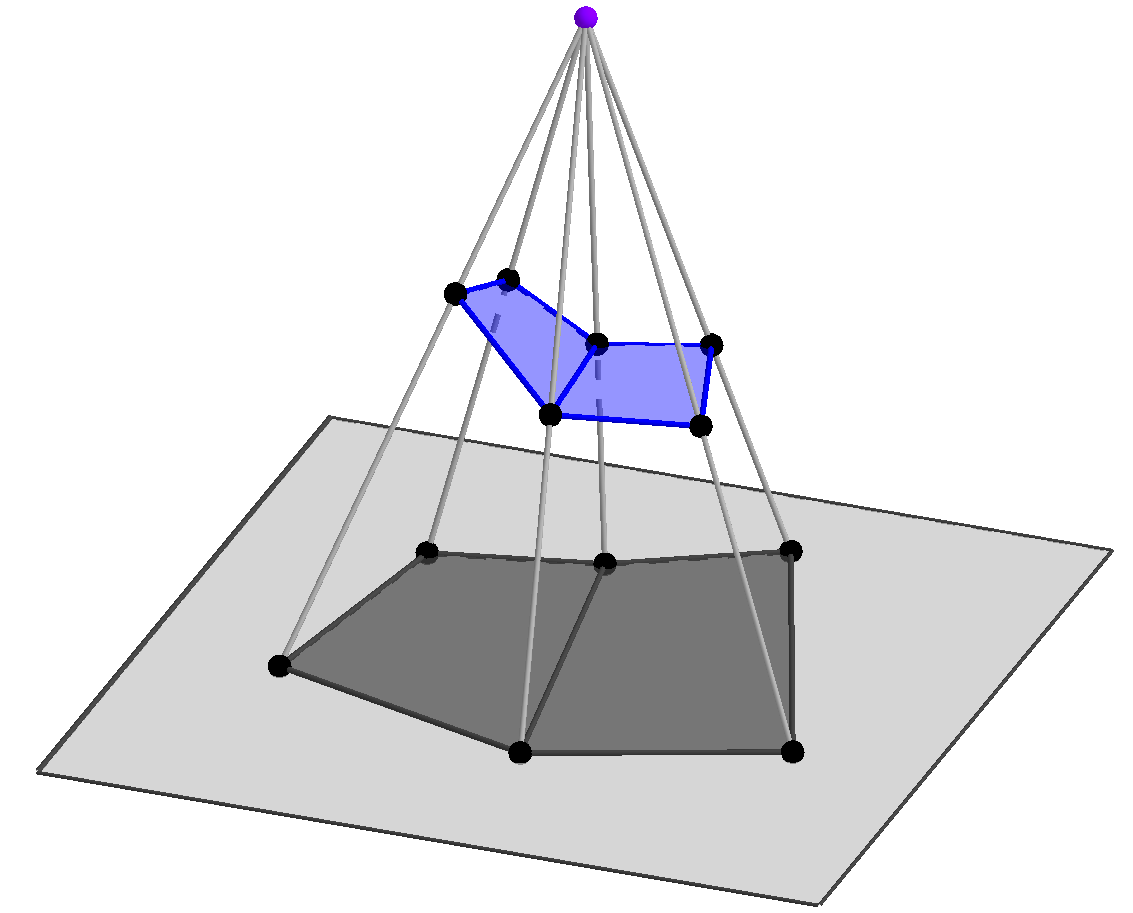}
	\caption{A lift $Q\colon \Sigma_{2,1} \mapsto \RP^3$ (blue) of a Q-net $P\colon \Sigma_{2,1} \mapsto \RP^2$ (black) with 0-dimensional center $C$.}
\end{figure}
Note that a non-degenerate Q-net $P$ defined on $\Sigma_{a,b}$ joins at most $a+b$ dimensions. This is clear because $P$ restricted to the set
\begin{align}
	\{(i,j) \mid ij = 0\},
\end{align}
joins at most $a+b$ dimensions and the remainder of the points of $P$ is in this join.

\begin{definition}\label{defn:extensive}
	We call a Q-net $P\colon \Sigma_{a,b} \rightarrow \RP^n$ \emph{extensive} if it is non-degenerate and if the image of $P$ joins an $(a+b)$-dimensional projective space. Moreover, for $c,d \in \N$ with $c \leq a, d \leq b$ we say $P$ is \emph{$\Sigma_{c,d}$-extensive} if every $\Sigma_{c,d}$ subpatch of $P$ is extensive.
\end{definition}

With this terminology, a non-degenerate Q-net is necessarily $\Sigma_{1,1}$-extensive.
Moreover, the restriction $P'$ of an extensive Q-net $P\colon \Sigma_{a,b} \rightarrow \RP^n$ to a subpatch $\Sigma_{a',b'} \subset \Sigma_{a,b}$ is also extensive. This can be proven by contradiction, because if the joined dimension of $P'$ is lower than $a'+b'$, then the joined dimension of $P$ has to be lower than $a + b$. 

\begin{lemma} \label{lem:lift}
	Let $a,b \in \N$, with $a,b>0$ and let $m=a+b$. Consider a non-degenerate Q-net $P\colon \Sigma_{a,b} \rightarrow \RP^{m}$ such that the points of $P$ join a subspace $E \subset \RP^{m}$ of dimension $n < m$. Let $C \subset \RP^{m}$ be supplementary to $E$ and let $\pi$ be the central projection with center $C$ to $E$. Then there is an extensive Q-net $\hat P \colon \Sigma_{a,b} \rightarrow \RP^{m}$ such that $\pi(\hat P) = P$. We call $\hat P$ a \emph{lift} of $P$.
\end{lemma}

\proof{
	We begin by choosing $\hat P(0,0)$ arbitrarily in $(P(0,0) \vee C ) \setminus C$. Then we iteratively choose $\hat P(i,0)$ in $(P(i,0) \vee C) \setminus C$ such that it is not in the join of the previously chosen points. After choosing $\hat P(a,0)$ we continue the same procedure with $\hat P(0,1), \hat P(0,2), \dots, \hat P(0,b)$. Thus, the points of $\hat P$ that we have chosen so far join all of $\RP^m$. Every other point $\hat P(i,j)$ with $ij\neq 0$ is uniquely determined by
	\begin{align}
		\hat P(i,j) = (P(i,j) \vee C) \cap (\hat P(i-1, j-1) \vee \hat P(i-1, j) \vee \hat P(i, j-1)).
	\end{align}
	The point $\hat P(i,j)$ is not in $C$ because $P$ is non-degenerate. Moreover, since $P$ is non-degenerate and $C$ is supplementary to $E$, $\tilde P$ is also non-degenerate. \qed
}

To simplify the phrasing of statements, if a Q-net $P$ is extensive we consider $\hat P = P$ to be a lift of $P$.

Since cross-ratios are invariant under projections, the Laplace invariants of a Q-net and its lift are the same. This implies that any lift of a BS-K{\oe}nigs net is a BS-K{\oe}nigs net and any lift of a D-K{\oe}nigs net is a D-K{\oe}nigs net. Moreover, it is not hard to see that the diagonal intersection net $\hat D$ of a lift $\hat P$ is also a lift of the diagonal intersection net $D$ of $P$.
Another useful and immediate observation is that for $P_i$ in the Laplace sequence of $P$ and $\hat P_i$ in the Laplace sequence of a lift $\hat P$, $\hat P_i$ is also a lift of $P_i$.

One reason to consider lifts is that extensive BS-K{\oe}nigs nets have the following useful property that is not necessarily true for non-extensive BS-K{\oe}nigs nets.

\begin{lemma} \label{lem:koenigslift}
	Let $a,b \in \N$ with $a,b \geq 1$ and let $m=a+b$. Consider an extensive BS-K{\oe}nigs net $P \colon \Sigma_{a,b} \rightarrow \RP^m$. Then, there are two different hyperplanes $U_1, U_2 \subset \RP^m$ such that
	\begin{align}
		P(i,j) \in \begin{cases}
			U_1 & \mbox{if } i+j\in 2\Z, \\
			U_2 & \mbox{if } i+j \in 2\Z + 1.
		\end{cases}
	\end{align}
	Conversely, if two such hyperplanes exist for an extensive Q-net $P$ then $P$ is a BS-K{\oe}nigs net.
\end{lemma}

\proof{
	We define
	\begin{align}
		U_1 &:= \join \{P(i,j) \mid \min(i,j) < 2 \mbox{ and } i+j \in 2\Z \},\\
		U_2 &:= \join \{P(i,j) \mid \min(i,j) < 2 \mbox{ and } i+j \in 2\Z + 1 \},
	\end{align}
	where $U_1$, $U_2$ are joined by $m$ points each. Assume $U_1$ or $U_2$ is not a hyperplane or that $U_1$ coincides with $U_2$. Then it follows that $U_1 \vee U_2$ is at most a hyperplane, which is a contradiction since $P$ is extensive.  Therefore, $U_1, U_2$ are two different hyperplanes.
	Subsequently, we use \cite[Thm 2.27]{BS2008DDGbook} which states that $P(i+2,j+2)$ is in the join of 
	\begin{align}
		P(i,j) \vee P(i+2,j) \vee P(i,j+2) \vee P(i+1,j+1),
	\end{align}
	if and only if $P$ is a BS-K{\oe}nigs net, from which both claims follow. \qed
}

In particular, Lemma~\ref{lem:koenigslift} applies to every lift of a BS-K{\oe}nigs net. Moreover, as a consequence of Lemma~\ref{lem:koenigslift}, if $P$ is an extensive BS-K{\oe}nigs net then the diagonal intersection net $D$ is contained in and joins $U_1 \cap U_2$.

%% file: laplaceinvariants.tex
\section{Laplace invariants and Laplace sequences} \label{sec:laplaceseq}

In this section we consider certain projective invariants of Q-nets (called \emph{Laplace invariants}), and show that they obey a recursion relation under Laplace transformations. We also gather several useful lemmas for terminating Laplace sequences.

Given a Q-net $P$ and its sequence of Laplace transforms $(P_i)_{i\in \Z}$, we let $H_i$ and $K_i$ denote the Laplace invariants of $P_i$. The next lemma shows that the sequence $(H_i)_{i\in \Z}$ is the same as the sequence $(K_{i+1})_{i\in \Z}$ up to a shift of index. 

\begin{lemma} \label{lem:hk}
	The Laplace invariants $H$ and $K$ satisfy the identities
	\begin{align} 
		K_1(i,j) & = H(i+1,j),\\
		H_{-1}(i,j) & = K(i, j+1).
	\end{align}
\end{lemma}
\proof{
	Follows immediately from the definition of the Laplace transforms and the Laplace invariants, that is from Definition~\ref{def:discreteLaplaceinvariants} and Equations~\eqref{eq:laplaceforward} and \eqref{eq:laplacebackward}. \qed
}

Moreover the sequence of Laplace invariants satisfies a recurrence relation.
This recurrence relation was found by Doliwa in \cite{doliwa1997geometricToda}, albeit using different conventions for the Laplace invariants.

\begin{theorem} \label{th:laplacerecurrence}
	The Laplace invariants $H$ satisfy the recurrence
	\begin{align} 
		H_1(i,j) &=  H_{-1}^{-1}(i,j)  \frac{1-H(i+1,j)}{1-H^{-1}(i,j)}  \frac{1-H(i,j+1)}{1-H^{-1}(i+1,j+1)}  	.	\label{eq:laplacerecurrence} 		
	\end{align}	
\end{theorem}

\proof{
	In \cite{abs} the proof of this formula was sketched from the view point of the octahedral lattice. Let us recall the key observations (in a more explicit way as compared to \cite{abs}). We assume without loss of generality that $i= j = 0$. Since Laplace invariants are preserved by projections, we work with a projection $P$ of the Q-net to $\RP^1$. 
	We use the notion of a \emph{multi-ratio}, which is a natural generalization of the cross-ratio for six points $A_1,A_2,\dots, A_6 \in \RP^1$, given by
	\begin{align}
		\mbox{mr}(A_1,A_2, A_3, A_4, A_5, A_6) = \frac{\det(A_1,A_2)}{\det(A_2,A_3)}\frac{\det(A_3,A_4)}{\det(A_4,A_5)}\frac{\det(A_5,A_6)}{\det(A_6,A_1)}.
	\end{align}
	We begin with the four multi-ratio identities 
	\begin{align}
		-1 &= \mbox{mr}(P_1(0,0), P_2(0,0), P_1(0,1), P_1(1,1), P(1,1), P_1(1,0)), \\
		-1 &= \mbox{mr}(P_1(0,1), P_2(-1,0), P_1(0,0), P_1(-1,0), P(0,1), P_1(-1,1)), \\
		-1 &= \mbox{mr}(P(0,1), P_{-1}(0,1), P(1,1), P(1,2), P_1(0,1), P(0,2)), \\
		-1 &= \mbox{mr}(P(1,1), P_{-1}(0,0), P(0,1), P(0,0), P_1(0,0), P(1,0)).
	\end{align}
	Each of these equations is an instance of Menelaus' theorem (see \cite{BS2008DDGbook}) applied to the four points and two focal points of a planar quad. On the other hand, we have the six cross-ratio identities
	\begin{align}
		H_1(0,0) &= \cro(P_1(0,0), P_2(0,0), P_1(0,1),P_2(-1,0)),\\
		H_{-1}(0,0) &= \cro(P(0,1), P_{-1}(0,1), P(1,1),P_{-1}(0,0)),\\
		1 - H^{-1}_1(0,0) &= \cro(P(0,0), P(0,1), P_1(-1,0),P_1(0,0)),\\
		1 - H^{-1}_1(1,1) &= \cro(P(1,1), P(1,2), P_1(0,1),P_1(1,1)),\\
		(1 -H _1(1,0))^{-1} &= \cro(P(1,0), P_1(0,0), P_1(1,0),P(1,1)),\\
		(1 - H_1(0,1))^{-1} &= \cro(P(0,1), P_1(-1,1), P_1(0,1),P(0,2)),
	\end{align}
	which follow from the definition of $H$ and the cross-ratio permutation rules. The product of the six cross-ratio equals the product of the four multi-ratios, as can be seen by an elementary expansion. Therefore the product of the six cross-ratios is 1. Thus, the product of the six quantities on the left hand side of the cross-ratio equations is also 1, and this is evidently equivalent to the claim of the theorem. \qed
	
}

\begin{remark}\label{rem:cluster}
	Note that with the substitution $H \mapsto -H$ the recurrence relation for Laplace invariants is evidently a special case of the \emph{cluster mutation} rule in the theory of cluster algebras. In other words, the sequence of Laplace invariants constitutes a \emph{Y-system}. See \cite{agprvrc, athesis} for more details.
\end{remark}

Using the identification of $H_{-1}$ and $K$ as in Lemma~\ref{lem:hk}, one immediately obtains the following form of the recurrence that is also commonly used in the literature.

\begin{corollary}\label{cor:laplacerecurrence}
	The Laplace invariants $H$ and $K$ satisfy the recurrences
	\begin{align} 
		H_1(i,j) &=  K^{-1}(i,j+1) \frac{1-H(i+1,j)}{1-H^{-1}(i,j)}  \frac{1-H(i,j+1)}{1-H^{-1}(i+1,j+1)},  \label{eq:H1formula}  \\
		K_{-1}(i,j) &= H^{-1}(i+1,j) \frac{1-K(i,j+1)}{1-K^{-1}(i,j)}  \frac{1-K(i+1,j)}{1-K^{-1}(i+1,j+1)}. \label{eq:K1formula}
	\end{align}	
\end{corollary}

In the following, for any map $P \colon \Sigma \to \RP^n$, we use the notation
\begin{align}
	\pah{P} (j) := \mathrm{join}\{P(i,j) \mid i : (i,j) \in \Sigma \}, \qquad \pav{P}(i) := \mathrm{join}\{P(i,j) \mid j : (i,j) \in \Sigma)\},
\end{align}
and we call each $\pah{P} (j)$ and each $\pav P(i)$ a \emph{parameter space} of $P$.

Let us also look at the interplay of algebra, in the sense of Laplace invariants, and geometry. In particular, even though the Laplace sequence is defined recursively, there is also an explicit construction as follows.

\begin{lemma}\label{lem:expllaplacetransform}
	Let $P\colon \Sigma_{m,m} \to \RP^{2m}$ be an extensive Q-net. Suppose that $P_m$ exists. Then
	\begin{align}
		P_m(0,0) = \bigcap_{k = 0}^{m} \bigvee_{\ell=0}^{m} P(k,\ell) = \bigcap_{k = 0}^{m} \pav{P}(k).
	\end{align}
\end{lemma}

\proof{
	We give a proof by induction over $m$. The base case $m=1$ is evident from the definition of the Laplace transform $P_1$. 
	Consider the four restrictions $\SW P$, $\NW P$, $\SE P$, $\NE P$ of $P$ to the respective patches
	\begin{align}
		&\{(i,j) \mid i,j\in \Sigma_{m-1}\}, & &\{(i,j+1) \mid i,j\in \Sigma_{m-1}\},\\
		&\{(i+1,j) \mid i,j\in \Sigma_{m-1}\}, & &\{(i+1,j+1) \mid i,j\in \Sigma_{m-1}\}.
	\end{align}
	Through induction we obtain
	\begin{align}
		P_{m-1}(0,0) &= \bigcap_{k = 0}^{m-1} \pav{\SW P}(k), & P_{m-1}(0,1) &= \bigcap_{k = 0}^{m-1} \pav{\NW P}(k), \\
		P_{m-1}(1,0) &= \bigcap_{k = 0}^{m-1} \pav{\SE P}(k), & P_{m-1}(1,1) &= \bigcap_{k = 0}^{m-1} \pav{\NE P}(k).
	\end{align}
	Let us denote
	\begin{align}
		L_0 &= P_{m-1}(0,0) \vee P_{m-1}(0,1), & L_1 &= P_{m-1}(1,0) \vee P_{m-1}(1,1).
	\end{align}
	The definition of the Laplace transform states that
	\begin{align}
		P_m(0,0) = L_0 \cap L_1.
	\end{align}
	By induction, we obtain the inclusions
	\begin{align}
		L_0 &= \bigcap_{k = 0}^{m-1} \pav{\SW P}(k) \vee \bigcap_{k = 0}^{m-1} \pav{\NW P}(k) \subset \bigcap_{k = 0}^{m-1} \pav{P}(k) =: E_0, \label{eq:laplinclusionone}\\
		L_1 &= \bigcap_{k = 1}^{m} \pav{\SE P}(k) \vee \bigcap_{k = 1}^{m} \pav{\NE P}(k) \subset \bigcap_{k = 1}^{m} \pav{P}(k) =: E_1.\label{eq:laplinclusiontwo}
	\end{align}
	Combined, we get that 
	\begin{align}
		P_m(0,0) = L_0 \cap L_1 \subset E_0 \cap E_1 = \bigcap_{k = 0}^{m} \pav{P}(k). \label{eq:laplaceproof}
	\end{align}
	Thus, we have shown one inclusion direction of the claim. For the other direction, we need to show that $E_0 = L_0$ and $E_1 = L_1$. Let us assume that $E_0 \neq L_0$, and therefore that $E_0$ is more than $1$ dimensional. Consider the join of the image of $\SW P$ which we denote $\SW \pa P$. As $P$ is extensive, the codimension of  $\SW \pa P$ is 2. Therefore, the intersection of $\SW \pa P$ with $E_0$ is at least 1-dimensional. On the other hand, the intersection of  $\SW \pa P$ with $E_0$ coincides with
	\begin{align}
		\bigcap_{k = 0}^{m-1} \pav{\SW P}(k),
	\end{align}
	which by induction is the point $P_{m-1}(0,0)$, a 0-dimensional space. Thus, we have arrived at a contradiction and $E_0$ cannot be more than 1-dimensional. Consequently, $E_0$ equals $L_0$, and analogously one shows that $E_1$ equals $L_1$. Inserting this into \eqref{eq:laplaceproof} yields the claim.\qed
}

Of course, using Lemma~\ref{lem:expllaplacetransform} and Definition~\ref{defn:laplace} immediately leads to the following corollary, which will be useful later on.

\begin{corollary}\label{cor:Laplaceterminatemsteps}
	Let $P\colon \Sigma \to \RP^n$ be a $\Sigma_{m,m}$-extensive Q-net such that $P_{m}$ exists. Then $P_m$ is Laplace degenerate if and only if
	\begin{align}
		\bigcap_{k = i}^{i+m} \bigvee_{\ell=j}^{j+m} P(k,l)
	\end{align}
	is a point that is independent of $i$ for all $j$.
\end{corollary}

See Figure~\ref{fig:koenigslaplace} for an illustration of Corollary~\ref{cor:Laplaceterminatemsteps} in the cases $m=1$ and $m=2$. Although the Q-net in Figure~\ref{fig:koenigslaplace} is not $\Sigma_{2,2}$-extensive, it can be lifted to a $\Sigma_{2,2}$-extensive net.

\begin{corollary}\label{cor:Laplacedegeneracy}
    Let $P\colon \Sigma \to \RP^n$ be a $\Sigma_{m,m}$-extensive Q-net such that $P_{m}$ exists, and let $d < m$. If the intersection
    \begin{align}
        \bigcap_{k = i}^{i+m} \bigvee_{\ell=j}^{j+m} P(k,l),
    \end{align}
    is a non-empty $d$-dimensional projective subspace that is independent of $i$ for all $j$, then $P_{m-d}$ is Laplace degenerate (assuming it exists).
\end{corollary}
\proof{
     The case $d=0$ is Corollary~\ref{cor:Laplaceterminatemsteps}. The case $d>1$ can be shown by induction. \qed
}

There is also a useful algebraic characterization of Q-nets that have a Laplace degenerate Laplace transform, which we give in the following lemma.

\begin{lemma}\label{lem:Laplaceterminatemstepsalgebraic}
	Let $P\colon \Sigma \to \R \mathrm{P}^n$ be a non-degenerate Q-net. Then,
	\begin{enumerate}
		\item $P_{1}$ is  Laplace degenerate if and only if $H(i,j) = 1$ for all $i,j$,
		\item $P_{-1}$ is Laplace degenerate if and only if $K(i,j) = 1$ for all $i,j$.
	\end{enumerate}
\end{lemma}
\begin{proof}
	Let us prove the claim for $P_1$. Since $H(i,j)$ is defined as a cross-ratio, $H(i,j) = 1$ is only possible if
	\begin{align}
		P(i,j) =  P(i,j+1) \qquad \mbox{or} \qquad P_{1}(i,j) = P_{1}(i-1,j).
	\end{align}
	However, $P(i,j)$ does not equal $P(i,j+1)$ because $P$ is a non-degenerate $Q$-net. Therefore $P_{1}(i,j)$ equals $P_{1}(i-1,j)$ and thus $P_1$ is Laplace degenerate. The proof for $P_{-1}$ is analogous.
\end{proof}

An immediate consequence of Lemma~\ref{lem:Laplaceterminatemstepsalgebraic} is that if $P_1$ is Laplace degenerate then every lift $\hat P_1$ is also Laplace degenerate, since the Laplace invariants do not change when lifting.
Next we consider a geometric characterization of Goursat degenerate Q-nets \cite{fairley2023thesis,bobenkofairley2023circularnets}.

\begin{lemma}\label{lem:Goursatterminatemsteps}
	Consider a $\Sigma_{m,m}$-extensive Q-net $P \colon \Sigma \to \RP^n$ such that $P_{m}$ exists and is nowhere Laplace degenerate. Then $P_{m}$ is Goursat degenerate if and only if the parameter spaces $\pav P(i)$ are $m$-dimensional.
\end{lemma}

\proof{
	Consider the spaces
   \begin{align}
		Z(i,j) := \bigvee_{\ell=j}^{j+m} P(i,l).
	\end{align}
	Assume the parameter spaces $\pav P(i)$ are $m$-dimensional. Consequently, $Z(i,j)$ is independent of $j$ since each $Z(i,j)$ is $m$-dimensional by $\Sigma_{m,m}$-extensivity of $P$. Then, using Lemma~\ref{lem:expllaplacetransform} we see that $P_m(i,j)$ is independent of $j$. As we assume that $P_m$ is nowhere Laplace degenerate, it follows that $P_m$ is Goursat degenerate. 	
    
	Conversely, assume $P_m$ is Goursat degenerate. Then, $Z(i,j)$ contains $P_m(i,k)$ for all $k$. Thus, $Z(i,j)$ and $Z(i,j+1)$ have $m+1$ points in common, namely 
	\begin{align}
		P(i,j +1), P(i,j +2), \dots, P(i,j+m) \quad \mbox{and}\quad P_m(i,j) = P_m(i,j+1).
	\end{align}
    By $\Sigma_{m,m}$-extensivity of $P$, the join of the $m+1$ points is $m$-dimensional.
	Therefore, $Z(i,j)$ and $Z(i,j+1)$ are equal to the same $m$-space. It follows that $Z(i,j)$ is independent of $j$. Then, each $\pav P(i)$ is  $m$-dimensional.
\qed
}

\begin{figure}
	\centering
	\includegraphics[width=.8\textwidth]{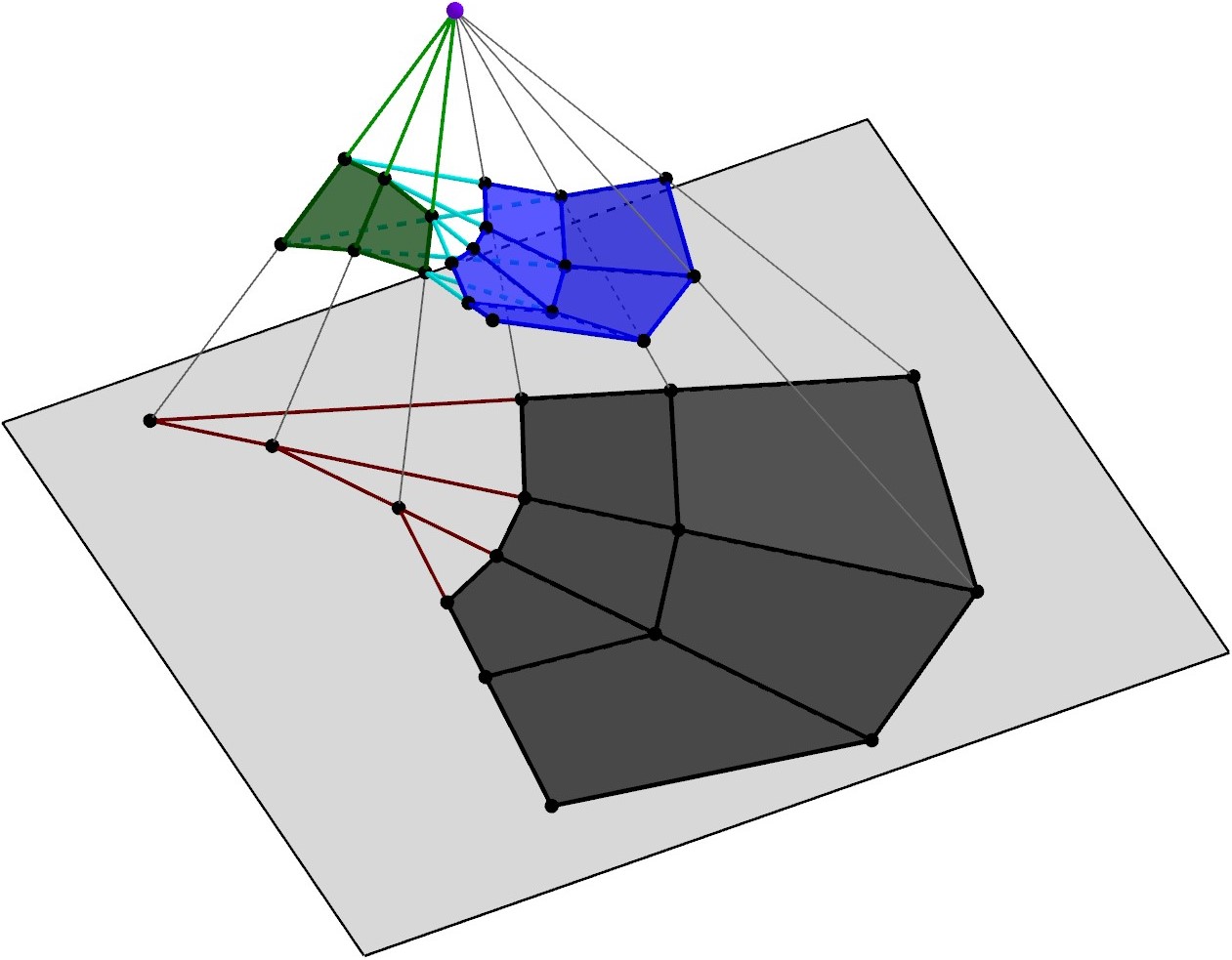}
	\caption{The lift $\hat P$ (blue) of a Q-net $P$ (black) with Goursat degenerate Laplace transform $P_1$. The lift $\hat P_1$ (green) is Laplace degenerate, with point of concurrency $\hat P_2$ equal to the center of projection $C$ (violet).}
	\label{fig:laplacelift}
\end{figure}

From Lemma~\ref{lem:Goursatterminatemsteps} we see that if $P$ is a Q-net with $m$-dimensional parameter lines then $P_m$ is Goursat degenerate. Note that we made no constraints on the parameter lines $\pah{P}(j)$ in Lemma~\ref{lem:Goursatterminatemsteps}.
Another consequence of Lemma~\ref{lem:Goursatterminatemsteps} is that the dimension that $P$ can span is limited if $P_m$ is Goursat degenerate. In particular, if $P$ is defined on $\Sigma_{a,b}$ then $P$ can be at most $\Sigma_{a,m}$-extensive. Recall that in Lemma~\ref{lem:lift} we showed that non-extensive Q-nets can be lifted to extensive Q-nets. However, in that lemma we assumed non-degeneracy. Next, we show how to lift Q-nets with a Goursat degeneration in the Laplace sequence.

\begin{lemma}
	Let $P \colon \Sigma_{a,b} \to \RP^n$ be a non-degenerate Q-net such that $P_1$ is Goursat degenerate, and let $\hat P$ be a lift of $P$. Then $\hat P_1$ is non-degenerate.
\end{lemma}
\proof{
	To show that the lift $\hat P_1$ is non-degenerate, it suffices to show the two conditions of Definition~\ref{def:nondegqnet} are satisfied for each quad. Thus, without loss of generality let us show the conditions for the quad involving the four points
	\begin{align}
		\hat P_1(0,0), \hat P_1(1,0), \hat P_1(1,1), \hat P_1(0,1).
	\end{align}
	Recall that by the definition of Goursat degeneracy we know that $P_1(0,0) \neq P_1(1,0)$. Moreover, because $P$ is non-degenerate neither of those two points coincides with $P(1,1)$. Consequently, the same holds for the lifted points, that is $\hat P_1(0,0)$, $\hat P_1(1,0)$ and $\hat P(1,1)$ are pairwise distinct. Also, by definition of the Laplace transform the four points
	\begin{align}
		\hat P(1,0), \hat P(1,1), \hat  P_1(0,0), \hat P_1(1,0),
	\end{align}
	are on a line, and the points
	\begin{align}
		\hat P(1,1), \hat P(1,2), \hat  P_1(0,1), \hat P_1(1,1),
	\end{align}
	are also on a line. Since $\hat P$ is extensive, the two lines need to be different (as otherwise $\hat P(1,0)$, $\hat P(1,1)$ and $\hat P(1,2)$ would not join a plane).
	As a further consequence, $\hat P_1(0,1)$ is not on the line joined by $\hat P_1(0,0)$ and $\hat P_1(1,0)$. As a result, the three points $\hat P_1(0,0)$, $\hat P_1(1,0)$ and  $\hat P_1(0,1)$ join a plane. This is enough to show that $\hat P_1$ is non-degenerate. \qed
}

Moreover, it turns out that Goursat degeneracy is actually closely related to Laplace degeneracy, as we show in the next lemma.

\begin{lemma}\label{lem:goursatliftstolaplace}
	Let $P\colon \Sigma_{a,b} \to \RP^n$ be a non-degenerate Q-net such that $P_1$ is Goursat degenerate, and let $\hat P$ be a lift of $P$. Then $\hat P_2$ is Laplace degenerate.
\end{lemma}

\proof{
	It suffices to prove the claim for a $\Sigma_{3,2}$ patch, thus we assume that $P$ is defined on  $\Sigma_{3,2}$. Since $P$ is non-degenerate and $P_1$ is Goursat degenerate, the four parameter spaces $\pav P(i)$ are 1-dimensional. On the other hand, the lift $\hat P$ is extensive and thus all four $\pav {\hat P}(i)$ are 2-dimensional. Recall that we obtain $P$ from $\hat P$ by a projection with center $C$.
	 Consequently, each plane $\pav {\hat P}(i)$ needs to intersect the center $C$ of the projection in a point $Q(i)$. Moreover, the three joins $\pav P(i) \vee \pav P(i+1)$ are 2-dimensional, but the lifts $\pav {\hat P}(i) \vee \pav {\hat P}(i+1)$ are 3-dimensional. Therefore, $\pav {\hat P}(i) \vee \pav {\hat P}(i+1)$ also intersects the center $C$ in a point. Hence, all four planes $\pav {\hat P}(i)$ intersect $C$ in the same point.
	Since $\hat P$ is extensive and $\hat P_1$ is non-degenerate note[AF]{by previous lemma}, $\hat P_2$ is determined by Lemma~\ref{lem:expllaplacetransform}, that is
	\begin{align}
		\hat P_2(0,0) &= \pav {\hat P}(0) \cap \pav {\hat P}(1) \cap \pav {\hat P}(2), & \hat P_2(1,0) &= \pav {\hat P}(1) \cap \pav {\hat P}(2) \cap \pav {\hat P}(3).
	\end{align}
	By the arguments above $\hat P_2(0,0)$ equals $\hat P_2(1,0)$, therefore $\hat P_2$ is Laplace degenerate.\qed

}

\begin{lemma}\label{lem:liftnotgoursat}
	Let $P\colon \Sigma_{a,b} \to \RP^n$ be a non-degenerate Q-net such that $P_1$ is Goursat degenerate, and let $\hat P$ be a lift of $P$. Then $\hat P_2(i,j) \neq \hat P_2(i,j+1)$ for all $i,j$.
\end{lemma}
\proof{
	It suffices to prove the claim for a $\Sigma_{2,3}$ patch. Moreover, it is not hard to see that if $\hat P_2(0,0) = \hat P_2(0,1)$ then $\hat P$ would join at most a 4-dimensional space and not a 5-dimensional space. This would be a contradiction with $\hat P$ being extensive. \qed
    }

\begin{remark}
	Note that in Figure~\ref{fig:laplacelift} we lift a Q-net $P$ defined on $\Sigma_{3,2}$ with Goursat degenerate $P_1$ from $\RP^2$ to $\RP^3$ instead of lifting it to $\RP^5$, since it is difficult to visualize more than three dimensions. If we would extend the Q-net $P$ to be defined on a larger $\Sigma_{3,3}$ patch, then one would observe that
	\begin{align}
		\hat P_2(0,0) &= \hat P_2(0,1), & \hat P_2(1,0) = \hat P_2(1,1),
	\end{align}
	in apparent contradiction to Lemma~\ref{lem:liftnotgoursat}. This is a general phenomenon if we consider a non-extensive lift to ``one dimension more'' of a Q-net $P$ with Laplace degenerate $\hat P_1$: $\hat P_2$ is degenerate of mixed type. However, this is not really a contradiction with Lemma~\ref{lem:liftnotgoursat} since in this lemma we assume that the lift is extensive.

\end{remark}

Let us make a practical algebraic observation on Goursat degenerate Q-nets. Of course, if a Q-net $P_m$ is Goursat or Laplace degenerate, then $P_{m+1}$ is not defined. Moreover, if $P_m$ is Laplace degenerate for some $m>0$, then $H_m$ is $1$ as stated in Lemma~\ref{lem:Laplaceterminatemstepsalgebraic}. Thus the recurrence formula for $H$ as given in Theorem~\ref{th:laplacerecurrence} is singular and $H_{m+1}$ is not defined, as expected. However, this is not the case for Goursat degenerate Q-nets. Instead, we arrive at the following lemma, which follows from Lemma~\ref{lem:goursatliftstolaplace}.

\begin{lemma}\label{lemma:HorKequals1Goursatdegenerate}
	Let $P\colon \Sigma \to \RP^n$ be a non-degenerate Q-net.
	\begin{enumerate}
		\item If $P_{1}$ is Goursat degenerate, then the right-hand side of \eqref{eq:H1formula} equals $1$ for all $(i,j)$. 
		\item If $P_{-1}$ is Goursat degenerate, then the right-hand side of \eqref{eq:K1formula} equals $1$ for all $(i,j)$.
	\end{enumerate}
\end{lemma}

\begin{proof}
	It suffices to prove the theorem under the assumption that $\Sigma$ is finite. Due to Lemma~\ref{lem:goursatliftstolaplace}, $P$ has a lift $\hat P$ such that $\hat P_2$ is Laplace degenerate. As a consequence of Lemma~\ref{lem:Laplaceterminatemstepsalgebraic} the Laplace invariants $\hat{H}_{1}(i,j)$ of the lift are equal to 1. Since the Laplace invariants of $P$ are the same as those of $\hat P$, we have shown that the right-hand side of equation \eqref{eq:H1formula} equals $1$. The claim for $P_{-1}$ follows analogously.
\end{proof}

Lemma~\ref{lemma:HorKequals1Goursatdegenerate} shows that -- from the perspective of Laplace invariants -- the condition that $P_{\pm 1}$ is Goursat degenerate is not distinguishable from the condition that $P_{\pm 2}$ is Laplace degenerate. 

Recall that given a BS-K{\oe}nigs net $P$, the diagonal intersection net $D$ is also a Q-net (and a D-K{\oe}nigs net) which therefore comes with its own sequence of Laplace invariants. As we show in the following theorem, it turns out that the sequence of Laplace invariants of $D$ coincides with that of $P$ -- but with reversed order. We use superscripts to indicate whether the Laplace invariants $H$ and $K$ are determined from $P$ or $D$.

\begin{theorem}\label{thm:symmetryLaplaceinvariants}
	Let $P$ be a BS-K{\oe}nigs net and let $D$ be the diagonal intersection net of $P$. For each $m \in \Z$ the Laplace invariants satisfy
	\begin{align*}
		H^{P}_{m}(i+1,j) = K^{D}_{-m}(i,j),\qquad H^D_{m}(i,j) = K^{P}_{-m}(i,j+1),
	\end{align*}
    assuming that $P_{\pm m}$ and $D_{\mp m}$ exist.
\end{theorem}

\begin{figure}[tb] 
	\centering
	\begin{tikzpicture}
		\small
		\coordinate[wvert, label={left:$P(1,0)$}] (p00) at (0,0) {};
		\coordinate[wvert, label={left:$P(1,1)$}] (p01) at (0,2) {};
		\coordinate[wvert, label={left:$P(1,2)$}] (p02) at (0,4) {};
		\coordinate[wvert, label={right:$P(2,0)$}] (p10) at (2,0) {};
		\coordinate[wvert, label={right:$P(2,1)$}] (p11) at (2,2) {};
		\coordinate[wvert, label={right:$P(2,2)$}] (p12) at (2,4) {};		
		\draw[-]
			(p00) -- (p01) -- (p02) -- (p12) -- (p11) -- (p10) -- (p00)
			(p01) edge[line width=2pt] (p11)
		;
		\coordinate[blvert, label={left:$D(0,0)$}] (d00) at (-1,1) {};
		\coordinate[blvert, label={below:$D(1,0)$}] (d10) at (1,1) {};
		\coordinate[blvert, label={right:$D(2,0)$}] (d20) at (3,1) {};
		\coordinate[blvert, label={left:$D(0,1)$}] (d01) at (-1,3) {};
		\coordinate[blvert, label={above:$D(1,1)$}] (d11) at (1,3) {};
		\coordinate[blvert, label={right:$D(2,1)$}] (d21) at (3,3) {};		
		\draw[-,blue]
			(d00) -- (d01) -- (d11) -- (d21) -- (d20) -- (d10) -- (d00)
			(d10) edge[line width=2pt] (d11)
		;		
	\end{tikzpicture}
	\hspace{5mm}
	\begin{tikzpicture}
		\small
		\coordinate[blvert, label={left:$D_1(0,-1)$}] (p00) at (0,0) {};
		\coordinate[blvert, label={left:$D_1(0,0)$}] (p01) at (0,2) {};
		\coordinate[blvert, label={left:$D_1(0,1)$}] (p02) at (0,4) {};
		\coordinate[blvert, label={right:$D_1(1,-1)$}] (p10) at (2,0) {};
		\coordinate[blvert, label={right:$D_1(1,0)$}] (p11) at (2,2) {};
		\coordinate[blvert, label={right:$D_1(1,1)$}] (p12) at (2,4) {};		
		\draw[-,blue]
			(p00) -- (p01) -- (p02) -- (p12) -- (p11) -- (p10) -- (p00)
			(p01) edge[line width=2pt] (p11)
		;
		\coordinate[wvert, label={left:$P_{-1}(0,0)$}] (d00) at (-1,1) {};
		\coordinate[wvert, label={below:$P_{-1}(1,0)$}] (d10) at (1,1) {};
		\coordinate[wvert, label={right:$P_{-1}(2,0)$}] (d20) at (3,1) {};
		\coordinate[wvert, label={left:$P_{-1}(0,1)$}] (d01) at (-1,3) {};
		\coordinate[wvert, label={above:$P_{-1}(1,1)$}] (d11) at (1,3) {};
		\coordinate[wvert, label={right:$P_{-1}(2,1)$}] (d21) at (3,3) {};		
		\draw[-,black]
			(d00) -- (d01) -- (d11) -- (d21) -- (d20) -- (d10) -- (d00)
			(d10) edge[line width=2pt] (d11)
		;		
	\end{tikzpicture}		
	\caption{The two Laplace invariants associated to the emphasized edges coincide on the left, and on the right as well. Specifically, on the left we have $K^P(1,1) = H^D(1,0)$ and on the right $H^P_{-1}(1,0) = K^D_{1}(0,0)$.}
\label{fig:symmetryLaplaceInvariants2}
\end{figure}
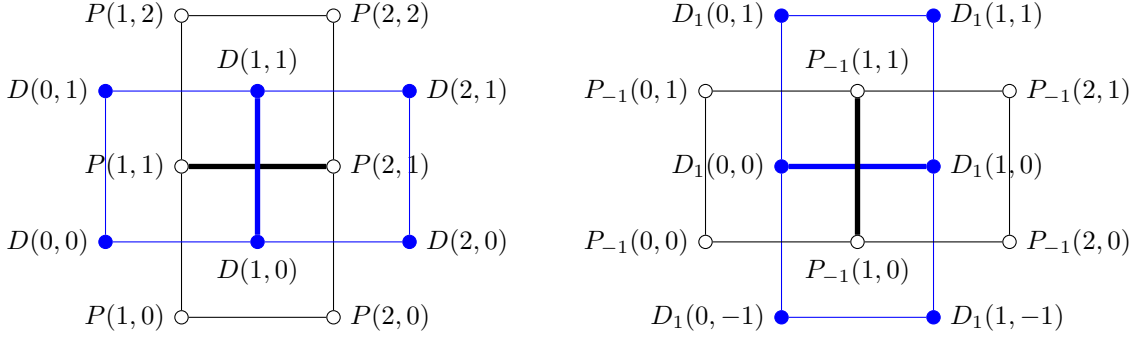

\begin{proof}
	Let us begin by proving the claim for $m=0$. We choose representative vectors $u,v,w,x \in \R^{n+1}\setminus \{0\}$ and scalars $\alpha, \beta \in \R$ such that 
	\begin{align*}
		P(i,j)&=[u], & P(i+1,j) &= [v], & P(i+1,j+1) &= [w], \\ 
		P(i,j+1) &= [u+v+w], & P(i+2,j) &= [x], & P(i+2,j+1) &= [\a v +\b w +x].
	\end{align*}
	Consequently, straight-forward computations show that 
	\begin{align*}
		D(i,j) &= [u+w], & P_1(i,j) &= [v+w], & D(i+1,j) &= [\b w+x], \\ 
		P_1(i+1,j) &= [\a v + \b w], & D_{-1}(i,j-1) &= [x-\b u], & D_{-1}(i,j) &= [-\a u +(\b -\a)w +x],
	\end{align*}
	where we used that \cite[Fig. 4]{BS2009Koenigsnets}
	\begin{align}
		D_{-1}(i,j-1) &= (D(i,j)\vee D(i+1,j)) \cap (P(i,j)\vee P(i+2,j)), \\
		D_{-1}(i,j) &= (D(i,j)\vee D(i+1,j)) \cap (P(i,j+1)\vee P(i+2,j+1)). \label{eq:LaplaceDinpahP}
	\end{align}	
	Inserting these expressions into the cross-ratios that define the Laplace invariants we obtain 
	\begin{align*}
		H^{P}(i+1,j) &= \mathrm{cr}([v], [\a v+\b w], [w], [v+w]) = \frac{\b}{\a},\\
		K^D(i,j)&= \mathrm{cr}([u+w], [-\a u + (\b-\a)w +x], [\b w+x], [x-\b u]) = \frac{\b}{\a}.
	\end{align*}	
	Clearly, we have shown $H^{P}(i+1,j) = K^D(i,j)$. In analogous fashion, one may compute that $H^D(i,j) = K^{P}(i,j+1)$. Hence, we have proven the claim for $m = 0$.
	
	To prove the claim for $m \neq 0$, recall that the Laplace invariants in the Laplace sequence satisfy a recurrence, see Corollary~\ref{cor:laplacerecurrence}. In particular, the Laplace invariants $H^P, K^P$ of $P$ determine the Laplace invariants of all $P_m$, and analogously for $D$. Moreover, one sees from Corollary~\ref{cor:laplacerecurrence} that the forwards and backwards recurrence for the Laplace invariants are the same. Since for $m=0$ the role of $H$ and $K$ in $P$ and $D$ is interchanged, we obtain the full claim of the theorem for all $m$.
\end{proof}

\begin{remark}
	Note that the index shifts in Theorem~\ref{thm:symmetryLaplaceinvariants} are an artifact of the $\Z^2$ coordinate notation we use. As visible in Figure~\ref{fig:symmetryLaplaceInvariants2}, the Laplace invariants that agree always belong to two edges that are combinatorial duals of each other. Via index shifts one may get at most half of the invariants to agree, for example if we define $E\colon \Z^2 \to \RP^n$ by $E(i,j)=D(i-1,j)$, then we get the formula $H^P_m(i,j) = K^E_{-m}(i,j)$. If we define $F\colon \Z^2 \to \RP^n$ by $F(i,j)=D(i,j-1)$, then we get the formula $H^F_m(i,j) = K^P_{-m}(i,j)$.
\end{remark}

\begin{remark}
    In differential geometry, the smooth Laplace invariants $h_i, k_i$ of the Laplace sequence of K{\oe}nigs nets satisfy $h_i = k_{-i}$ for all $i \in \Z$, see \cite{Tzitzeica1924geometrie}. Thus, the symmetry in Theorem~\ref{thm:symmetryLaplaceinvariants} is in agreement with the smooth theory.
\end{remark}

\begin{remark}
	Theorem~\ref{thm:symmetryLaplaceinvariants} was independently discovered in a slightly more general setup (of the so called \emph{K{\oe}nigs binets}), as will be shown in an upcoming publication \cite{adtbinets}.
\end{remark}

The symmetry in the Laplace invariants of $P$ and $D$ has the following consequences for Laplace degenerations.

\begin{proposition} \label{prop:bsdlaplace}
	Let $P$ be a BS-K{\oe}nigs net and let $D$ be its diagonal intersection net. Then $P_{m}$ is Laplace degenerate if and only if $D_{-m}$ is Laplace degenerate (assuming $P_m$ and $D_{-m}$ exist).
\end{proposition}
\proof{
	By Lemma~\ref{lem:Laplaceterminatemstepsalgebraic}, $P_m$ is Laplace degenerate if and only if $H_{m-1}^P =1$, which by Theorem~\ref{thm:symmetryLaplaceinvariants} is equivalent to  $K_{-m+1}^D =1$. By Lemma~\ref{lem:Laplaceterminatemstepsalgebraic}, this is equivalent to $D_{-m}$ being Laplace degenerate.\qed
}

\begin{proposition}\label{prop:bsdgoursat}
	Let $P$ be a BS-K{\oe}nigs net and let $D$ be its diagonal intersection net. If $P_{m}$ is Goursat degenerate, then $D_{-m-1}$ is Laplace degenerate (assuming $D_{-m-1}$ exists).
\end{proposition}
\proof{
    It suffices to consider a finite patch of $P$ with size $\Sigma_{m+2,m+3}$ (so that the corresponding patch of $D$ has size $\Sigma_{m+1,m+2}$).
	By Lemma~\ref{lem:goursatliftstolaplace}, any lift $\hat P$ 
    of $P$ is such that $\hat P_{m+1}$ is Laplace degenerate. Hence it follows from Proposition~\ref{prop:bsdlaplace} that $\hat D_{-m-1}$ is Laplace degenerate, which implies that $D_{-m-1}$ is Laplace degenerate as well.\qed
}

Combined, Proposition~\ref{prop:bsdlaplace} and Proposition~\ref{prop:bsdgoursat} provide a proof of Theorem~\ref{thm:symmetryLaplacedegenerateBSandD}.

\FloatBarrier

%% file: terminatingsequences.tex
\section{Terminating Laplace sequences}\label{section:terminating}\label{sec:terminating}

In this section we show that Laplace (and Goursat) terminating Laplace sequences of Kœnigs nets are finite. We begin with a proof using Laplace invariants and the symmetry between the Laplace invariants of $P$ and $D$ shown in Theorem~\ref{thm:symmetryLaplaceinvariants}. 
In the following, if a Laplace transform $P_m$ Laplace degenerates to a curve, we simply denote by $P_m(j)$ the point $P_m(0,j)$ which equals $P_m(i,j)$ for all $i$.  If $P_{-m}$ is Laplace degenerate, we write $P_{-m}(i) = P_{-m}(i,0)$.

\begin{theorem}\label{thm:BSandDiagonoalIntersectionNetcompatible}
	Let $P\colon \Sigma \to \RP^n$ be a BS-Kœnigs net such that $P_{m}$ is Laplace degenerate. Suppose that $P_{-m-1}$ and $D_{-m}$ exist. Then, $P_{-m-1}$ and $D_{-m}$ are Laplace degenerate and $P_{-m-1}(i) = D_{-m}(i)$ for all $i \in \Z$, see Figure~\ref{fig:koenigslaplace}.
\end{theorem}

\begin{proof}
It suffices to prove the theorem locally for $$P\colon  \Sigma_{m+1,m+2} \to \RP^n \text{ and } D\colon \Sigma_m \times \{-1,\ldots, m+2\} \to \RP^n.$$ By Proposition~\ref{prop:bsdlaplace}, $D_{-m}$ is Laplace degenerate because $P_{m}$ is Laplace degenerate.  Without loss of generality, we assume that $P$ and $D$ are extensive. By Corollary~\ref{cor:Laplaceterminatemsteps}, the spaces $\pah{D}(j) $ for $j = -1, \ldots, m+2$ are concurrent at $D_{-m}(0)$ because $D_{-m}$ is Laplace degenerate. By Equation~\eqref{eq:LaplaceDinpahP}, we see that $D_{-1}(i,j)$ is contained in $\pah{P}(j+1)$. Then, $D_{-m}(0)$ is contained in $\pah{P}(j+1)$. Thus, $D_{-m}(0)$ is contained in $\bigcap^{m+2}_{j=0} \pah{P}(j)$. By Corollary~\ref{cor:Laplaceterminatemsteps}, $P_{-m-1}$ is Laplace degenerate.
\end{proof}

\begin{figure}[tb] 
	\centering
	\includegraphics[width=0.38\textwidth]{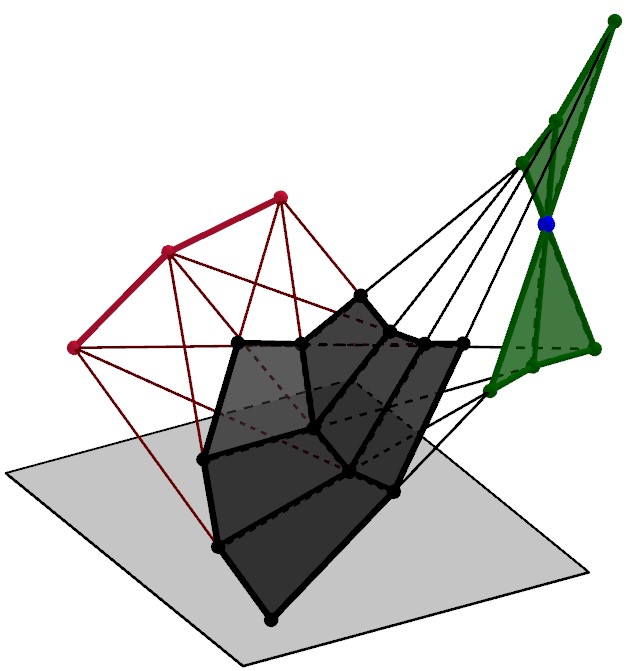}
	\includegraphics[width=0.61\textwidth]{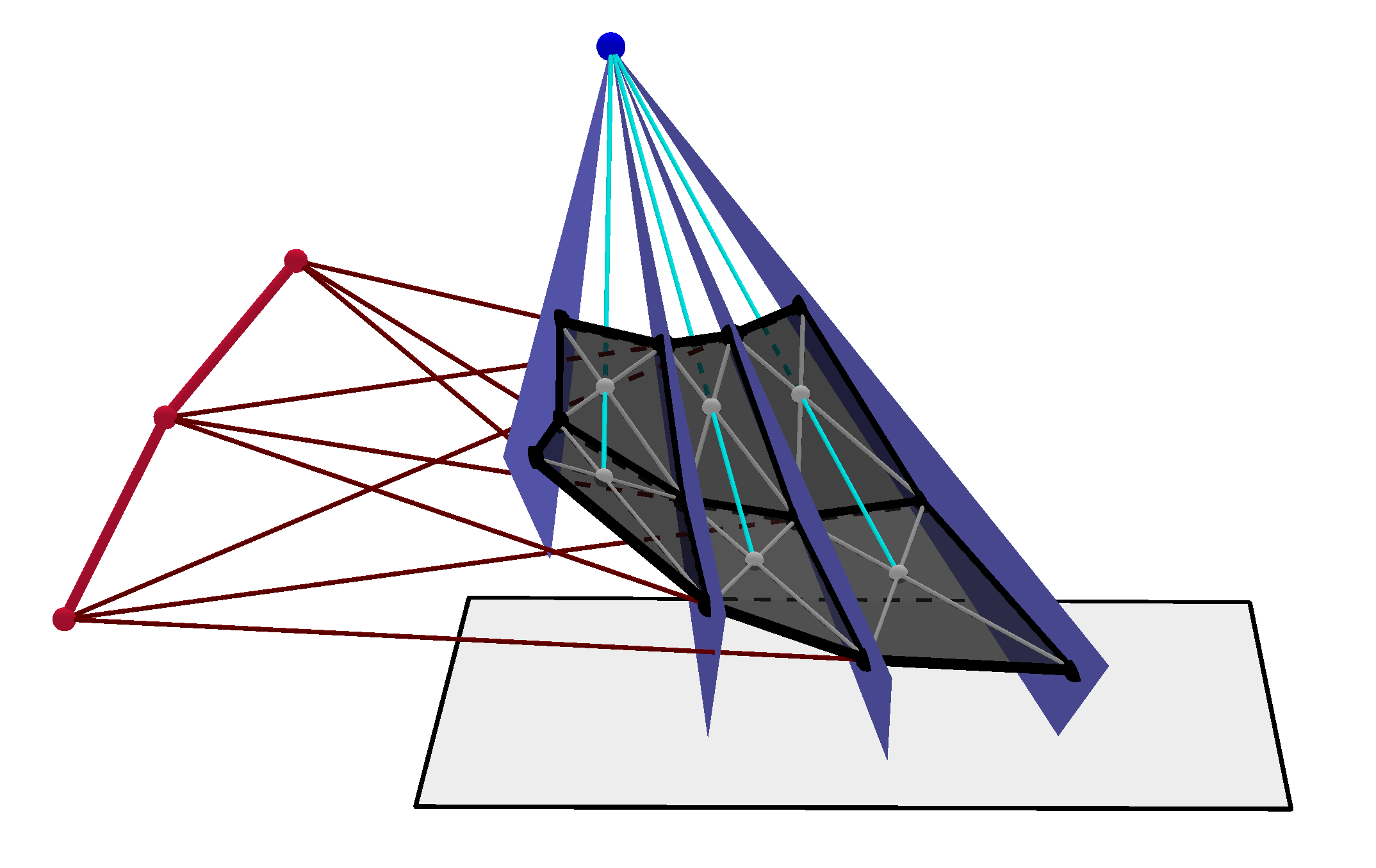}
	\caption{A BS-Kœnigs net $P \colon \Sigma_{2,3} \to \RP^3$ (black) with Laplace degenerate $P_1$ (red). By Theorem~\ref{thm:BSandDiagonoalIntersectionNetcompatible}, $D_{-1}$ and $P_{-2}$ (blue point) are equal and Laplace degenerate. The four blue planes $\pah P(j)$ are concurrent at $P_{-2}$.}
	\label{fig:koenigslaplace}
\end{figure}

Although we think that the proof of Theorem~\ref{thm:BSandDiagonoalIntersectionNetcompatible} is quite nice, the claim of Theorem~\ref{thm:BSandDiagonoalIntersectionNetcompatible} involves an additional assumption compared to our main theorem (Theorem~\ref{thm:BSandDKoenigslaplacedegunsymmetric}): that $D_{-m}$ exists. To be able to drop this assumption we need a slightly different proof strategy, which we present in the following. This alternative strategy builds on prior work \cite{fairley2023thesis, bobenkofairley2023circularnets} concerning Laplace sequences of Q-nets that are inscribed in quadrics.

Recall that due to Lemma~\ref{lem:koenigslift} a (finite) extensive BS-Kœnigs net is contained in a pair of hyperplanes $U_1, U_2$. In this section it turns out to be practical to view
\begin{align}
	\mathcal U := U_1 \cup U_2 \label{eq:uquadric}
\end{align}
as a degenerate quadric. The singular points of this quadric are the points in $U_1 \cap U_2$.
The plan is to apply the following theorem \cite{fairley2023thesis, bobenkofairley2023circularnets}, which we recall without proof.

\begin{theorem}\label{thm:Qnetinquadricconjugatelaplacepoints}
	Let $m,n\in \N$ with $n \geq 2$ and let $\mathcal Q \subset \RP^n$ be a quadric. Consider a Q-net $P\colon \Sigma_{m,m} \to \RP^n$ such that
	\begin{enumerate}
		\item the Laplace transforms $P_{-m}$, $P_{m}$ are well-defined, 
		\item and all points of $P$ except maybe $P(m,m)$ are in $\mathcal Q$.
	\end{enumerate}

	Then $P(m,m)$ is in $\mathcal Q$ if and only if $P_{m}(0,0)$ is conjugate to $P_{-m}(0,0)$ with respect to $\mathcal Q$.
\end{theorem}

Note that it is permissible that the quadric $\mathcal{Q}$ in Theorem~\ref{thm:Qnetinquadricconjugatelaplacepoints} is degenerate, hence it applies to extensive BS-Kœnigs nets (and thus to the lifts of BS-Kœnigs nets) with $\mathcal Q = \mathcal U$.
This observation is the key idea for the proof of Theorem~\ref{thm:BSandDKoenigslaplacedegunsymmetric}. In a sense, the realization that BS-Kœnigs nets lift to BS-Kœnigs nets inscribed in a quadric allows us to proceed analogously to \cite{fairley2023thesis,bobenkofairley2023circularnets}, where it is shown that Q-nets that are inscribed in quadrics have the property that the Laplace sequence terminates in one direction if and only if it terminates in both directions. 

We begin with two lemmas that we need for our new proof strategy.

\begin{lemma}\label{lemma:nosingularpointslaplacesequence}
	Let $P\colon \Sigma_{m,m} \to \RP^{2m}$ be an extensive BS-Kœnigs net, and let $\mathcal U$ be the degenerate quadric  -- as in Equation~\eqref{eq:uquadric} -- that contains $P$. Suppose that the Laplace transforms $P_{m}$ and $P_{-m}$ are well defined. Then, $P_{m}(0,0)$ and $P_{-m}(0,0)$ are not singular points of $\mathcal U$.
\end{lemma}

\begin{proof}
	We give a proof by contradiction, that is we assume that $P_{m}(0,0)$ is a singular point.
	Define a non-degenerate Q-net $O\colon \Sigma_{m,m} \to \RP^{2m}$ such that:
	\begin{enumerate}
		\item $O(0,0)$ is a point on the line spanned by $P(0,0)$ and $P(0,1)$ not equal to $P(0,0)$, $P(0,1)$ or $P_1(0,0)$, 
		\item $O$ agrees with $P$ everywhere else,
		\item and $O_{-m}(0,0)$ is well-defined.
	\end{enumerate}
	Such a Q-net exists because the condition that $O_{-m}(0,0)$ is well-defined is an open condition. Moreover, $O$ is non-degenerate by construction. As a result of the definition of $O$, the iterated Laplace transforms of $P$ and $O$ are equal in the positive direction, that is $P_k = O_k$ for all $k > 0$. Consequently, $O_{m}(0,0)$ is also a singular point. Therefore $O_{-m}(0,0)$ is conjugate to $O_{m}(0,0)$ because every point is conjugate to a singular point. Thus, by Theorem~\ref{thm:Qnetinquadricconjugatelaplacepoints}, we get that $O(0,0)$ is contained in the quadric $\mathcal{U}$.  Therefore, the whole line $P(0,0) \vee P(0,1)$ is contained in $\mathcal{U}$, because this line contains the three points $P(0,0), P(0,1)$ and $O(0,0)$ that are also in $\mathcal U$. However, for the lift of a BS-Kœnigs net this is impossible, as $P(0,0)$ and $P(0,1)$ are always in different hyperplanes of $\mathcal{U}$. Thus we have found a contradiction, and therefore $P_{m}(0,0)$ cannot be a singular point. The argument for $P_{-m}(0,0)$ is analogous.
\end{proof}

\begin{lemma}\label{lemma:laplacesingularpointinline}
	Let $P\colon  \Sigma_{m,m+1} \to \RP^{2m+1}$ be an extensive BS-Kœnigs net, and let $\mathcal U$ be the degenerate quadric that contains $P$. Suppose that the Laplace transforms $P_{m}$ and $P_{-m}$ are well defined and that the two points of $P_{m}$ are distinct. The line $P_{m}(0,0) \vee P_{m}(0,1)$ contains a singular point of $\mathcal{U}$ if and only if  $P_{-m}(0,0) = P_{-m}(0,1)$.
\end{lemma}

\begin{proof}
	``$\Rightarrow$'' Consider the two joins
	\begin{align}
		L_m &:= P_{m}(0,0) \vee P_{m}(0,1), & L_{-m} := P_{-m}(0,0) \vee P_{-m}(0,1).
	\end{align}
	Let $S$ be the singular point contained in $L_m$. Let us assume that $P_{-m}(0,0)$ does not equal $P_{-m}(0,1)$, hence $L_{-m}$ is a line. First, let us show that with this assumption $L_m$ is conjugate to $L_{-m}$. By Theorem~\ref{thm:Qnetinquadricconjugatelaplacepoints}, $P_m(0,0)$ is conjugate to $P_{-m}(0,0)$ and $P_m(0,1)$ is conjugate to $P_{-m}(0,1)$. Moreover, since $S$ is a singular point, $S$ is conjugate to all points in $\RP^{2m+1}$. In particular, $S$ is conjugate to both $P_{-m}(0,0)$ and $P_{-m}(0,1)$ and therefore to all of $L_{-m}$. Moreover, due to Lemma~\ref{lemma:laplacesingularpointinline}, $P_m(0,0)$ and $P_m(0,1)$ are non-singular and therefore do not coincide with $S$. Since there are now two points on $L_m$ that are conjugate to $P_{-m}(0,0)$ and two points that are conjugate to $P_{-m}(0,1)$, we obtain that $L_m$ is indeed conjugate to $L_{-m}$. Now, define a non-degenerate Q-net $O: \Sigma_{m,m+1} \to \RP^{2m+1}$ with the same conditions as in the proof of Lemma~\ref{lemma:nosingularpointslaplacesequence}, and therefore with the same properties. We observe that

	\begin{align}
		L_{-m} &= P_{-m}(0,0) \vee P_{-m}(0,1) =  P_{-m+1}(0,1) \vee P_{-m+1}(1,1) \\
		&= O_{-m+1}(0,1) \vee O_{-m+1}(1,1) = O_{-m}(0,0) \vee O_{-m}(0,1),
	\end{align}
	where the second line follows because $P_{-m+1}(0,1)$ equals $O_{-m+1}(0,1)$ and $P_{-m+1}(1,1)$ equals $O_{-m+1}(1,1)$, and the other equalities follow from the definition of the Laplace transform. Hence, $O_{-m}(0,0)$ is on $L_{-m}$. Additionally, $O_m(0,0)$ is on $L_m$, since $O_m(0,0)$ equals $P_m(0,0)$. Therefore $O_{-m}(0,0)$ is conjugate to $O_{m}(0,0)$, which means that the conditions of Theorem~\ref{thm:Qnetinquadricconjugatelaplacepoints} are met and thereby $O(0,0)$ is contained in $\mathcal U$. Thus we obtain the same contradiction as in the end of the proof of Lemma~\ref{lemma:laplacesingularpointinline}. Hence,  $P_{-m}(0,0)$ equals $P_{-m}(0,1)$.
	
	``$\Leftarrow$'' As a consequence of Theorem~\ref{thm:Qnetinquadricconjugatelaplacepoints}, $P_m(0,0)$ and $P_m(0,1)$ are contained in the polar of the point $P_{-m}(0,0) = P_{-m}(0,1)$. By Lemma~\ref{lemma:nosingularpointslaplacesequence}, $P_{-m}(0,0)$ is not a singular point. Therefore, the polar of $P_{-m}(0,0)$ is a hyperplane, which necessarily contains $U_1 \cap U_2$, that is the codimension 2 space of singular points of $\mathcal U$. Since $P_m(0,0)$ and $P_{m}(0,1)$ span a line in the polar of $P_{-m}(0,0)$, this line has to intersect $U_1 \cap U_2$ in a point, which is a singular point.
\end{proof}

Note that the condition that $P_{-m}(0,0) = P_{-m}(0,1)$ in Lemma~\ref{lemma:laplacesingularpointinline} is equivalent to saying that $P_{-m}$ is Laplace degenerate.
We are now ready to prove Theorem~\ref{thm:BSandDKoenigslaplacedegunsymmetric}.
For improved readability, we split it into four propositions.

\begin{proposition} \label{prop:bskoenigslaplace}
	Let $P\colon \Sigma \to \RP^n$ be a BS-Kœnigs net. If $P_{m}$ is Laplace degenerate, then $P_{-(m+1)}$ is Laplace degenerate (assuming it exists). 
\end{proposition}

\proof{
	To show that $P_{-m-1}$ is Laplace degenerate, it suffices to consider a finite patch of $P$ with domain $\Sigma_{m+1,m+2}$ since we only have to show that $P_{-m-1}(i,j)$ coincides with $P_{-m-1}(i,j+1)$ for all $i,j$. Moreover, it suffices to consider the case that $P$ is extensive. If $P$ is not extensive we may instead consider the lift of $P$ as in Lemma~\ref{lem:lift}. It follows from Lemma~\ref{lemma:laplacesingularpointinline} that the lines
	\begin{align}\label{eq:linesoflifts}
		P_m(0) \vee P_m(1), \quad P_m(1) \vee P_m(2),
	\end{align}	
	do not contain a singular point, since $P_{-m}$ is nowhere Laplace degenerate. Consequently, the polar spaces $P_m^\perp(0)$, $P_m^\perp(1)$ do not coincide, nor do the polar spaces $P_m^\perp(1)$ and $P_m^\perp(2)$ coincide. Due to Theorem~\ref{thm:Qnetinquadricconjugatelaplacepoints} the point $P_{-m}(0,j)$ and the point $P_{-m}(1,j)$ are in $P_m^\perp(j)$ for each $j \in \{0,1,2\}$. Hence the line 
	\begin{align}
		\ell(j) := P_{-m}(0,j) \vee P_{-m}(1,j)
	\end{align}
	is also in $P_m^\perp(j)$ for each $j$. Moreover, the intersection of $\ell(j)$ and $\ell(j+1)$ is contained in the intersection of $P_m^\perp(j)$ and $P_m^\perp(j+1)$ which is the locus of singular points of $\mathcal U$. Thus, the intersection of $\ell(j)$ and $\ell(j+1)$ is a singular point. Now, the definition of the Laplace transform implies that
	\begin{align}
		P_{-m-1}(0,j) = \ell(j) \cap \ell(j+1).
	\end{align}
	Therefore the points of $P_{-m-1}$ are singular. If the point $P_{-m-1}(0,0)$ were different from $P_{-m-1}(0,1)$, this would imply that there are multiple singular points on the line $\ell(1)$, which would therefore be an isotropic line. But then that would imply that $P_{-m}(0,1)$ and $P_{-m}(1,1)$ (both on $\ell(1)$) are singular, which is not possible due to Lemma~\ref{lemma:nosingularpointslaplacesequence}. Hence the points $P_{-m-1}(0,0)$ and $P_{-m-1}(0,1)$ coincide, which shows that $P_{-m-1}$ is Laplace degenerate. \qed
}

\begin{remark}
    Proposition~\ref{prop:bskoenigslaplace} does not exclude the possibility that $P_m$ is degenerate of mixed type (see Remark~\ref{rem:mixedtype}). Therefore, let us note that the proof is valid even in this case. Indeed, since we are working with an (extensive) lift in the proof of Proposition~\ref{prop:bskoenigslaplace}, Lemma~\ref{lem:liftnotgoursat} ensures that the lines in Equation~\eqref{eq:linesoflifts} are well-defined.
\end{remark}

We now turn to finishing the proof of Theorem~\ref{thm:BSandDKoenigslaplacedegunsymmetric}.

\begin{proposition}\label{prop:BSGoursatpositiveandLaplacenegative}
	Let $P\colon \Sigma \to \RP^n$ be a BS-Kœnigs net. If $P_{m}$ is Goursat degenerate, then $P_{-m-2}$ is Laplace degenerate (assuming it exists).
\end{proposition}

\proof{
	It suffices to consider $P$ defined on a patch of size $\Sigma_{m+2,m+3}$. Then, $P_{-m-2}$ consists of two points which we show coincide. Consider a lift $\hat P$ of $P$. Due to Lemma~\ref{lem:goursatliftstolaplace}, $\hat P_{m+1}$ is Laplace degenerate. Hence, Proposition~\ref{prop:bskoenigslaplace} implies that $\hat P_{-m-2}$ is Laplace degenerate. Consequently, $P_{-m-2}$ is also Laplace degenerate.\qed
}

\begin{proposition} \label{prop:dkoenigslaplace}
	Let $D\colon \Sigma \to \RP^n$ be a D-Kœnigs net. If $D_{m}$ is Laplace degenerate, then $D_{-m-1}$ is Laplace degenerate (assuming it exists). 
\end{proposition}

\proof{
	It was shown in \cite{Steinmeier2018} that there is a (non-unique) BS-Kœnigs net $P$ such that $D$ is the diagonal intersection net of $P$. Since $D_{m}$ is Laplace degenerate, Theorem~\ref{thm:symmetryLaplacedegenerateBSandD} implies that $P_{-m}$ is Laplace degenerate as well. Subsequently Proposition~\ref{prop:bskoenigslaplace} implies that $P_{m+1}$ is Laplace degenerate. Finally, we use Theorem~\ref{thm:symmetryLaplacedegenerateBSandD} again which implies that $D_{-m-1}$ is Laplace degenerate.\qed
}

\begin{proposition} \label{prop:dkoenigsgoursat}
	Let $D\colon \Sigma \to \RP^n$ be a D-Kœnigs net. If $D_{m}$ is Goursat degenerate, then $D_{-m-2}$ is Laplace degenerate (assuming it exists).
\end{proposition}
\proof{
	We follow the usual arguments. There is a lift $\hat D$ of $D$ such that $\hat D_{m+1}$ is Laplace degenerate due to Lemma~\ref{lem:goursatliftstolaplace}, and then the claim follows from Proposition~\ref{prop:dkoenigslaplace}.\qed
}

All combined, Propositions~\ref{prop:bskoenigslaplace} - \ref{prop:dkoenigsgoursat} provide a proof of one of our main theorems, namely Theorem~\ref{thm:BSandDKoenigslaplacedegunsymmetric}.

Recall that we have already discussed right after Theorem~\ref{thm:BSandDKoenigslaplacedegunsymmetric} what may appear as a paradox here. On a more technical level, note that we have shown in Lemma~\ref{lemma:nosingularpointslaplacesequence} that the points of $P_m$ are \emph{not} singular with respect to $\mathcal U$. On the other hand, in the proof of Proposition~\ref{prop:bskoenigslaplace} we see that the points of $P_{-(m+1)}$ \emph{are} singular. Thus, there really is a qualitative difference between $P_m$ and $P_{-(m+1)}$.

We also need to make sure that Proposition~\ref{prop:bskoenigslaplace} is not a void statement, which it would be if $P_{-(m+1)}$ never exists. In the next section we show how to construct such nets, and we show that restrictive conditions are needed so that $P_{-m}$ is Laplace degenerate (instead of $P_{-(m+1)}$).

%% file: symterminatingsequences.tex
\section{Symmetrically terminating Laplace sequences} \label{sec:symterminating}

In this section, we show how to construct Kœnigs nets such that both $P_m$ \emph{and} $P_{-m}$ are Laplace degenerate. First, with Lemma~\ref{lem:Koenigslaplacedegsymmetric} we provide a fundamental lemma which couples forwards and backwards Laplace degeneracy. This lemma is a generalization of the corresponding lemma in \cite{KMT2023conenets} which established the $m=1$ case that is shown in Figure~\ref{fig:koenigsdoublelaplace}. Then we show how to construct BS-Kœnigs net from initial data using Lemma~\ref{lemma:BSKoenigsline}. This lemma provides the basis for Lemma~\ref{lem:BSKoenigsPmLaplacedeg} which shows how to construct BS-Kœnigs nets from initial data such that one direction is Laplace degenerate after $m$ steps. Subsequently, we combine Lemma~\ref{lem:Koenigslaplacedegsymmetric} and  Lemma~\ref{lem:BSKoenigsPmLaplacedeg} to obtain Lemma~\ref{lem:BSKoenigsDoubleLaplacedeg}, which enables us to show Theorem~\ref{th:bskoenigsdoublelaplace}, which states how to construct BS-Kœnigs nets from initial data such that \emph{both} directions are Laplace degenerate after $m$ steps.

\begin{figure}[tb] 
	\centering
	\includegraphics[width=0.6\textwidth]{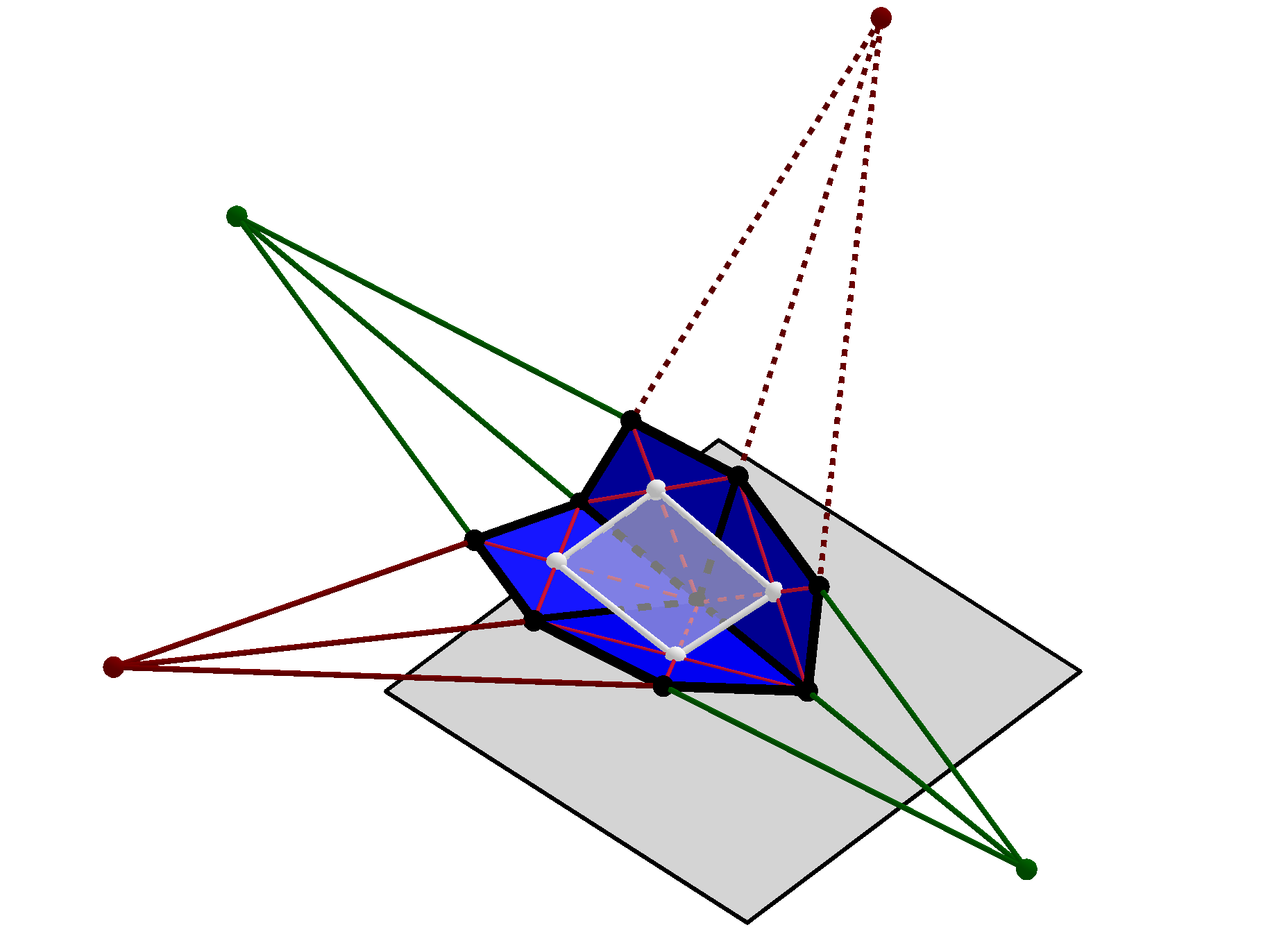}
	\caption{A BS-Kœnigs net $P$ (blue) with Laplace degenerate $P_1$ (green) and Laplace degenerate $P_{-1}$ (brown). We see that $P$ is a BS-Kœnigs net because the diagonal intersection points are in a plane (white). By Lemma~\ref{lem:Koenigslaplacedegsymmetric}, the three triples of concurrent solid lines imply the concurrency of the three dashed lines.}
	\label{fig:koenigsdoublelaplace}
\end{figure}

\begin{lemma}\label{lem:Koenigslaplacedegsymmetric}
	Let $P\colon \Sigma \to \RP^n$ be a BS-Kœnigs net such that $P_m$ is Laplace degenerate. If $P_{-m}(0, j)$ is independent of $j$, then $P_{-m}(i,j)$ is independent of $j$ for all $i$.
\end{lemma}

\begin{proof}
	It suffices to prove the theorem for a $\Sigma_{m+1,m+1}$ patch. Moreover, we may assume that $P$ is extensive, otherwise we consider a lift of $P$. Consider the line
	\begin{align}
		L_m := P_m(0,0) \vee P_{m}(0,1).
	\end{align}
	Since the theorem assumes that $P_{-m}(0,0)$ equals $P_{-m}(0,1)$, Lemma~\ref{lemma:laplacesingularpointinline} implies that $L_m$ contains a singular point. On the other hand, since $P_m$ is Laplace degenerate we observe that
	\begin{align}
		L_m = P_m(1,0) \vee P_{m}(1,1).
	\end{align}
	Thus, by applying Lemma~\ref{lemma:laplacesingularpointinline} in the other direction, we conclude that $P_{-m}(1,0)$ equals $P_{-m}(1,1)$. Iterating this argument proves the claim.		
\end{proof}

\begin{lemma}\label{lemma:BSKoenigsline}
	Consider a  BS-Kœnigs net $P\colon \Sigma_{2,2} \to \RP^n$ and a non-degenerate Q-net $P'\colon \Sigma_{2,2} \to \RP^n$ that agrees with $P$ everywhere except on $(2,2)$. Then $P'$ is a BS-Kœnigs net if and only if $P'(2,2)$ is on the line joining $P(1,1)$ and $P(2,2)$.
\end{lemma}

\begin{proof}
	We assume that $P$ is extensive, otherwise we consider a lift of $P$. Therefore, $P$ and $P'$ span $\RP^4$. As $P$ and $P'$ are Q-nets both $P(2,2)$ and $P'(2,2)$ are contained in the plane
	\begin{align}
		E = P(1,1) \vee P(2,1) \vee P(2,2).
	\end{align}
	
	By Lemma~\ref{lem:koenigslift}, $P'$ is a BS-Kœnigs net if and only if $P'(2,2)$ is contained in the 3-space
	\begin{align}
		F = P(0,0) \vee P(1,1) \vee P(2,0) \vee P(0,2).
	\end{align}
	Also by Lemma~\ref{lem:koenigslift}, $P(2,2)$ is contained in $F$ because $P$ is a BS-Kœnigs net. So, the intersection $E\cap F$ is the line joining $P(1,1)$ and $P(2,2)$.
\end{proof}

Informally, Lemma~\ref{lemma:BSKoenigsline} states that if we know a $\Sigma_{2,2}$ patch of a BS-Kœnigs net $P$ except for $P(2,2)$, then all the possible values of $P(2,2)$ are on a line that passes through $P(1,1)$.

Consider a Q-net $P^0$ defined on a horizontal and a vertical strip of quads (Figure~\ref{fig:boundarydata}, left), that is
\begin{align}
	P^0 \colon (\Sigma_1 \times \Z) \cup (\Z \times \Sigma_1) \to \RP^n.
\end{align} 
If we want to find a BS-Kœnigs net $P$ defined on all of $\Z^2$ that coincides with $P^0$ wherever $P^0$ is defined, Lemma~\ref{lemma:BSKoenigsline} implies that there is one free parameter per additional quad. The next lemma enables us to refine this extension procedure to obtain BS-Kœnigs nets with terminating Laplace sequences.

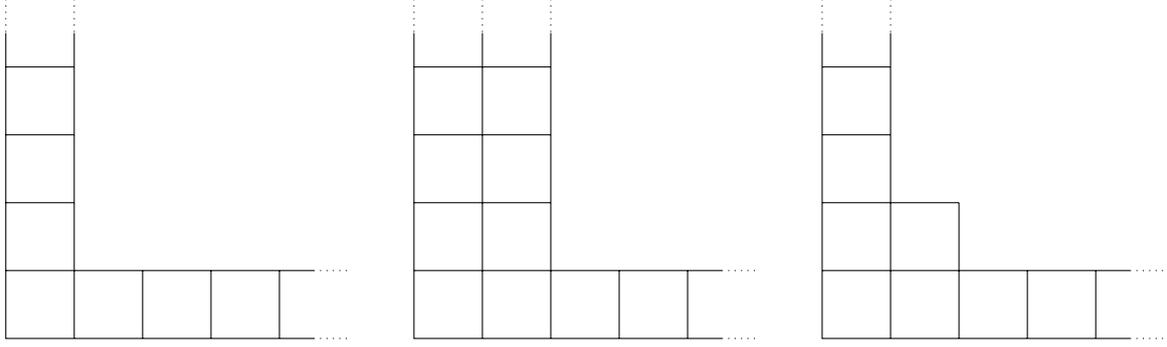
\begin{figure}
	\centering
	\begin{tikzpicture}[scale=0.9]
		\draw[-]
			(0,4.5) -- (0,0) -- (4.5,0)
			(0,1) -- (4.5,1)	(1,0) -- (1,4.5)
			(0,2) -- (1,2) (0,3) -- (1,3) (0,4) -- (1,4)
			(2,0) -- (2,1) (3,0) -- (3,1) (4,0) -- (4,1)
		;
		\draw[dotted]
			(0,5) -- (0,4.5) (1,5) -- (1,4.5)
			(4.5,0) -- (5,0) (4.5,1) -- (5,1)
		;		
	\end{tikzpicture}
	\hspace{6mm}
	\begin{tikzpicture}[scale=0.9]
		\draw[-]
			(0,4.5) -- (0,0) -- (4.5,0)
			(0,1) -- (4.5,1) (1,0) -- (1,4.5)
			(2,0) -- (2,4.5)
			(0,2) -- (2,2) (0,3) -- (2,3) (0,4) -- (2,4)
			(3,0) -- (3,1) (4,0) -- (4,1)
		;
		\draw[dotted]
			(0,5) -- (0,4.5) (1,5) -- (1,4.5) (2,5) -- (2,4.5)
			(4.5,0) -- (5,0) (4.5,1) -- (5,1)
		;		
	\end{tikzpicture}
	\hspace{6mm}
	\begin{tikzpicture}[scale=0.9]
		\draw[-]
			(0,4.5) -- (0,0) -- (4.5,0)
			(0,1) -- (4.5,1)	(1,0) -- (1,4.5)
			(0,2) -- (2,2) (0,3) -- (1,3) (0,4) -- (1,4)
			(2,0) -- (2,2) (3,0) -- (3,1) (4,0) -- (4,1)
		;
		\draw[dotted]
			(0,5) -- (0,4.5) (1,5) -- (1,4.5)
			(4.5,0) -- (5,0) (4.5,1) -- (5,1)
		;		
	\end{tikzpicture}
	\caption{Left: boundary data to construct a BS-Kœnigs net $P$, one degree of freedom remains for each non-boundary point. Center: boundary data if $P$ is additionally such that $P_2$ is Laplace degenerate, no further degrees of freedom remain. Right: boundary data if $P$ is additionally such that $P_{-2}$ is also Laplace degenerate. }
	\label{fig:boundarydata}
\end{figure}

\begin{lemma}\label{lem:BSKoenigsPmLaplacedeg}
	Consider a BS-Kœnigs net $P\colon \Sigma_{m+1,m} \to \RP^n$ with $m\geq 2$, and such that $P_{m}$ exists. Then, there is a unique BS-Kœnigs net $P': \Sigma_{m+1,m} \to \RP^n$ that agrees with $P$ everywhere except possibly on $(m+1,m)$, such that $P'_{m}$ is Laplace degenerate.
\end{lemma}

\proof{
	We assume that $P$ is extensive. Otherwise we consider a lift of $P$, prove the claim for the lift, and then project back down. For $P'_m$ to be Laplace degenerate, Corollary~\ref{cor:Laplaceterminatemsteps} implies that $P'(m+1,m)$ needs to be in the $m$-space $Q$ that is the join of $P_m(0,0)$ and 
	\begin{align}
		P(m+1,0), P(m+1,1), \dots, P(m+1,m-1).
	\end{align}
	On the other hand, Lemma~\ref{lemma:BSKoenigsline} shows that $P'(m+1,m)$ is on the line $\ell$ through $P(m+1,m)$ and $P(m,m-1)$. The line $\ell$ is in the join of $\pav P(m)$ and $\pav P(m+1)$, which is an $(m+1)$-dimensional space. Moreover, $\ell$ is not contained in $\pav P(m)$, as otherwise $P(m+1,m)$ would also be in $\pav P(m)$, which in turn would be a contradiction with the extensivity of $P$. Therefore $\ell$ and the $m$-space $Q$ intersect in a point, which is $P'_m(m+1,m)$ -- which is therefore unique.\qed
}

In particular, consider a BS-Kœnigs net $P^0$ defined on the following ``broader'' horizontal and vertical strips of quads (see Figure~\ref{fig:boundarydata}, center, for the case $m=2$):
\begin{align}
	P^0 \colon (\Sigma_{m} \times \Z) \cup (\Z \times \Sigma_{m-1}) \to \RP^n. \label{eq:initialdatabslaplace}
\end{align}
Such a Q-net $P^0$ may be constructed using Lemma~\ref{lemma:BSKoenigsline}.
If we want to find a BS-Kœnigs net $P$ defined on all of $\Z^2$ that coincides with $P^0$ wherever $P^0$ is defined \emph{and} such that $P_m$ is Laplace degenerate, Lemma~\ref{lem:BSKoenigsPmLaplacedeg} implies that there is \emph{no} free parameter per additional quad. Then, $P$ is uniquely defined by the data on the broader strips.

\begin{lemma}\label{lem:BSKoenigsDoubleLaplacedeg}
	Consider a BS-Kœnigs net $P \colon \Sigma_{m+1,m+1} \to \RP^n$ with $m\geq2$, and such that $P_{m}$ exists. Then, there is a unique BS-Kœnigs net $P': \Sigma_{m+1,m+1} \to \RP^n$ that agrees with $P$ everywhere except possibly on $(m+1,m)$, $(m,m+1)$ and $(m+1,m+1)$, such that $P'_{m}$ and $P'_{-m}$ are Laplace degenerate.
\end{lemma}

\proof{
	We assume that $P$ is extensive. Otherwise we consider a lift of $P$, prove the claim, and project back down. Lemma~\ref{lem:BSKoenigsPmLaplacedeg} to shows that $P'(m+1,m)$ is uniquely determined from the requirement that $P'_m(0,0) = P'_m(1,0)$, which follows from the requirement that $P'_m$ is Laplace degenerate. The backwards version of Lemma~\ref{lem:BSKoenigsPmLaplacedeg} shows that $P'(m,m+1)$ is uniquely determined from the requirement that $P'_{-m}(0,0) = P'_{-m}(0,1)$. Subsequently, employing Lemma~\ref{lem:BSKoenigsPmLaplacedeg} again shows that $P'(m+1,m+1)$ is uniquely determined from the requirement that $P'_m(0,1) = P'_m(1,1)$. Finally, Lemma~\ref{lem:Koenigslaplacedegsymmetric} shows that $P'_{-m}(1,0) = P'_{-m}(1,1)$. Thus, $P'_{-m}$ is also Laplace degenerate, which proves the claim.\qed
}

Consider a BS-Kœnigs net $P^0$ defined on the following horizontal and vertical strips of quads and one more point (Figure~\ref{fig:boundarydata}, right):
\begin{align}
	P^0 \colon (\Sigma_{m-1} \times \Z) \cup (\Z \times \Sigma_{m-1}) \cup \{(m,m)\} \to \RP^n. \label{eq:initialdatadoublebs}
\end{align} 
If we want to find a BS-Kœnigs net $P$ defined on all of $\Z^2$ that coincides with $P^0$ wherever $P^0$ is defined and such that \emph{both} $P_m$ and $P_{-m}$ are Laplace degenerate, Lemma~\ref{lem:BSKoenigsDoubleLaplacedeg} implies that there is \emph{no} free parameter per additional quad. Therefore $P$ is uniquely defined by the data given by $P^0$. Since we consider this one of our main results, let us put it into a theorem.

\begin{theorem}\label{th:bskoenigsdoublelaplace}
	For $m\geq 2$, BS-Kœnigs nets $P$ such that both $P_m$ and $P_{-m}$ are Laplace degenerate exist and are determined by initial data as given in \eqref{eq:initialdatadoublebs}.
\end{theorem}

The reason that we have excluded the $m=1$ case so far is that this case is slightly different. In particular, Lemma~\ref{lem:BSKoenigsPmLaplacedeg} and Lemma~\ref{lem:BSKoenigsDoubleLaplacedeg} do not hold for $m=1$. The reason for this is that there is no BS-Kœnigs condition on $\Sigma_{1,1}$ patches. On the other hand, the $m=1$ case was investigated in \cite{KMT2023conenets}. There, the authors showed that instead Theorem~\ref{th:bskoenigsdoublelaplace} holds for initial data 
\begin{align}
	P^0 \colon (\Sigma_{1} \times \Z) \cup (\Z \times \Sigma_{1}) \to \RP^n,
\end{align} 
which is not the same as the initial data in Equation~\eqref{eq:initialdatadoublebs} for $m=1$.
Additionally, they showed that any Q-net $P$ such that both $P_1$ and $P_{-1}$ are Laplace degenerate is automatically a BS-Kœnigs net. This is the reason that the initial data differs from the $m \geq 2$ cases.

Comparing Equation~\eqref{eq:initialdatabslaplace} and Equation~\eqref{eq:initialdatadoublebs}, we see that there are more degrees of freedom if we only require that $P_{m}$ is Laplace degenerate, compared to if we require that both $P_{m}$ and $P_{-m}$ are Laplace degenerate. This shows that if $P$ is a BS-Kœnigs net such that $P_{m}$ is Laplace degenerate, then generically $P_{-m}$ is not Laplace degenerate. Rather, the generic case is as in Theorem~\ref{thm:BSandDKoenigslaplacedegunsymmetric}: $P_{-m-1}$ is Laplace degenerate if $P_m$ is Laplace degenerate.

\begin{remark}
    One may also ask whether there are D-Kœnigs nets $D$ -- instead of BS-Kœnigs nets -- such that both $D_m$ and $D_{-m}$ are Laplace degenerate. The existence of these follows readily from Theorem~\ref{th:bskoenigsdoublelaplace}. Indeed, by Lemma~\ref{lem:BSKoenigsandDiagonalIntersectionNetisDKoenigs}, the diagonal intersection net $D$ of a BS-Kœnigs net $P$ is a D-Kœnigs net, and Proposition~\ref{prop:bsdlaplace} shows that $D_m$ is Laplace degenerate if and only if $P_m$ is Laplace degenerate. Therefore, the diagonal intersection net $D$ of a BS-Kœnigs net with Laplace degenerate $P_m$ and $P_{-m}$  also has Laplace degenerate $D_m$ and $D_{-m}$.
\end{remark}

\section{Concluding remarks}\label{sec:outlook}

\subsection{Discrete isothermic surfaces}

Let us briefly discuss how our results may apply to so called \emph{discrete isothermic surfaces} \cite{bpdisosurfaces}. Although the original characterization is different, one can show  \cite{BS2008DDGbook} that discrete isothermic surfaces are BS-Kœnigs nets that are also \emph{circular nets}. A circular net is a Q-net $P$ such that the points 
\begin{align}
    P(i,j), P(i+1,j), P(i,+1,j+1), P(i,j+1),    
\end{align}
are contained in a circle for all $(i,j)\in \Z^2$.
Let $P$ be a discrete isothermic surface in $\R^3$, and assume that $P_{-m}$ is Laplace degenerate.
From our results for BS-Kœnigs nets (specifically Theorem~\ref{thm:BSandDKoenigslaplacedegunsymmetric}), it follows that if $P_{m+1}$ exists, then $P_{m+1}$ is Laplace degenerate as well. On the other hand, from similar results for circular nets \cite[Theorem~3.5]{bobenkofairley2023circularnets} it follows that if $P_{m+3}$ were to exist, then $P_{m+3}$ would be Goursat degenerate. Thus, our result is ``stronger'' in the sense that our results imply an earlier termination of the Laplace sequence for discrete isothermic surfaces. However, since discrete isothermic nets are a special case of BS-Kœnigs nets, it is possible that the Laplace sequence may be even shorter. We consider this is an interesting open problem. 

Next, we want to consider an even more special case of isothermic surfaces, but first, we take a step back. Let us assume that $P$ is a circular net in $\R^3$ with one family of spherical curvature lines -- which is a circular net such that for all $j$ there is a sphere $\mathcal S(j)$ such that $P(i,j) \in \mathcal S(j)$. It turns out that an immediate consequence of this property is that $P_{-3}$ is Goursat degenerate \cite{bobenkofairley2023circularnets}. Then, a result of \cite{bobenkofairley2023circularnets} shows that $P_3$ is Laplace degenerate.

However, it was also shown in \cite{bobenkofairley2023circularnets}, that a better discretization of the corresponding smooth surfaces is obtained by considering the special case of circular nets with one family of spherical curvature lines such that $P_{2}$ is already Laplace degenerate (instead of $P_3$). 

Next, let us additionally assume that $P$ is a BS-Kœnigs net, that is $P$ is a discrete isothermic surface with one parameter family of spherical curvature lines, and such that $P_2$ is Laplace degenerate. Then, Theorem~\ref{thm:BSandDKoenigslaplacedegunsymmetric} implies that $P_{-3}$ is not just Goursat degenerate but also Laplace degenerate, and hence mixed type degenerate. Geometrically, this implies that all spheres $\mathcal S(j)$ belong to a 3-dimensional sphere pencil. That is, there is a sphere (or point, or imaginary sphere) $\mathcal X$, such that all spheres $\mathcal S(j)$ intersect $\mathcal X$ orthogonally.

Finally, from the viewpoint of smooth Kœnigs net it may make sense to put one more constraint on $P$, namely that $P_{-2}$ is already Laplace degenerate (instead of $P_{-3}$). Geometrically, this would imply that all spheres $\mathcal S(j)$ belong to a 2-dimensional sphere pencil. That is, there is a pair of points that is contained by all spheres $\mathcal S(j)$ (or the spheres all have a common orthogonal circle, or all the spheres touch a common line in a common point). However, it is not completely clear if such discrete isothermic surfaces can be constructed and how.

\begin{remark}
    There is also a recent paper \cite{hoffmannSzewieczek2024isothermic} investigating a different special class of discrete isothermic surfaces with a family of spherical curvature lines. We strongly suspect that these discrete surfaces also have terminating Laplace sequences, and also that the spheres $\mathcal S(j)$ all belong to a 2-dimensional sphere pencil.
\end{remark}

\begin{remark}
     Another possibility, which has not yet been investigated, would be to define discrete isothermic surfaces as circular nets that are also D-Kœnigs nets (instead of BS-Kœnigs nets). If $P_m$ Laplace degenerates for such a net, after how many steps $m'$ is $P_{-m'}$ Laplace degenerate?
\end{remark}

\subsection{CKP Darboux maps}

As mentioned in Remark~\ref{rem:cluster}, the Laplace invariants of a Q-net can be viewed as cluster variables, and Laplace transforms act as cluster mutations. There are two well known subvarieties (algebraic constraints) of cluster variables on $\Z^2$. The so called \emph{resistor subvariety} and the \emph{Ising subvariety}. Q-nets with cluster variables in the resistor subvariety are exactly the BS-Kœnigs nets (see~\cite{agprvrc}) and the D-Kœnigs nets (see \cite{athesis}), so the two types of Kœnigs nets that we investigated in the paper. This provokes the question: what about Q-nets with cluster variables in the Ising subvariety? These were shown in \cite{agprvrc} to be the so called \emph{CKP Darboux maps} \cite{schiefdarboux}. A Darboux map is a map $M$ defined on the edges of $\Z^2$, such that the four points around a quad of $\Z^2$ are contained in a line. Let $P$ be the restriction of $M$ to vertical edges, and $P_1$ the restriction of $M$ to horizontal edges. It is not hard to see that $P_1$ is indeed the Laplace transform of $P$. In this manner, a Darboux map defines a unique Laplace sequence. Moreover, it is known that the cluster variables in the Ising subvariety have a certain symmetry. Translated to Laplace invariants, this symmetry reads as $H_k = H_{-k}$ for all $k$, up to a shift of indices. Hence, if $P_m$ is Laplace degenerate, and if $P_{-m}$ exists, then $P_{-m}$ is also Laplace degenerate (due to Lemma~\ref{lem:Laplaceterminatemstepsalgebraic}).

\subsection{Periodic discrete Kœnigs nets}

The projection of a Laplace sequence to $\RP^1$ is a solution of the so called \emph{dSKP equation} \cite{athesis}. This allows to apply results on singularities of the dSKP equation \cite{adtmdskp, adtmdskpgeometry}. In particular, one gets the following statement: If $P$ is a $m$-periodic Q-net such that $P_{-1}$ is either Laplace degenerate (or Goursat degenerate), and if $P_{m-1}$ exists, then $P_{m-1}$ is also Laplace degenerate (or Goursat degenerate) -- a consequence of \cite[Theorem~1.6 with $p=1$]{adtmdskp}. It would be interesting to know what happens with terminating Laplace sequences for periodic discrete Kœnigs nets.

\subsection{Other singularities}

In this paper we dealt with Laplace sequences that terminated because the final Kœnigs net was Laplace or Goursat degenerate. However, a Kœnigs-net can also have local singularities. These cause the Laplace transformation to be undefined only in local patches of the Q-net. For example, it is possible that $P_k(i,j) = P_k(i+1,j)$ for only one pair $(i,j)$. We suspect that one such forward singularity is not enough to force a backwards singularity somewhere. If so, the question remains: how many local forward singularities are needed to force a backwards singularity?